\documentclass[reqno]{amsart}
\usepackage{amsmath,amssymb,amsfonts,amsthm}
\usepackage{mathrsfs}
\usepackage{xcolor}
\usepackage{enumitem}
\usepackage[colorlinks=false,pdfborder={0 0 0}]{hyperref}
\numberwithin{equation}{section}
\newtheorem{theorem}{Theorem}[section]
\newtheorem{proposition}[theorem]{Proposition}
\newtheorem{lemma}[theorem]{Lemma}
\newtheorem{corollary}[theorem]{Corollary}
\theoremstyle{definition}
\newtheorem{definition}[theorem]{Definition}
\theoremstyle{remark}
\newtheorem{remark}[theorem]{Remark}
\theoremstyle{plain}
\newtheorem*{thmdim}{Theorem~\ref{thm:dim-formula}}

\DeclareMathOperator{\diver}{div}
\DeclareMathOperator{\supp}{supp}
\DeclareMathOperator{\tr}{tr}
\DeclareMathOperator{\spann}{span}
\DeclareMathOperator{\rank}{rank}
\DeclareMathOperator{\im}{im}
\DeclareMathOperator{\ind}{Ind}
\newcommand{\Real}{\operatorname{Re}}
\newcommand{\Imag}{\operatorname{Im}}

\newcommand{\R}{\mathbb R}
\newcommand{\Z}{\mathbb Z}
\newcommand{\C}{\mathbb C}
\newcommand{\HH}{\mathcal H}
\newcommand{\X}{\mathfrak X}
\newcommand{\p}{\partial}
\newcommand{\dd}{\,\mathrm{d}}
\newcommand{\Sig}{\Sigma}
\newcommand{\Om}{\Omega}
\newcommand{\dOm}{\partial\Omega}
\newcommand{\dSig}{\partial\Sigma}
\newcommand{\Sgb}{\overline{\Sigma}}
\newcommand{\dSgb}{\partial\overline{\Sigma}}
\newcommand{\Sdouble}{\widehat{\Sigma}}
\newcommand{\Ind}{\ind}
\newcommand{\II}{\mathrm{I\!I}}
\newcommand{\Lsq}{L^2\HH^1_T}
\newcommand{\Lrho}{L^2_\rho\HH^1_T}
\newcommand{\Lstar}{L^2_\ast\HH^1_T}
\newcommand{\Lst}{\mathcal L^\star}
\newcommand{\Ed}{\mathcal E}
\newcommand{\cQ}{\mathcal Q}
\newcommand{\cJ}{\mathcal J}
\newcommand{\cB}{\mathcal B}
\newcommand{\lcut}{\widehat\varphi}

\usepackage{marginnote}

\begin{document}

\title[Complete noncompact free boundary minimal surfaces]{Index estimates for complete noncompact free boundary minimal surfaces}

\author[Cavalcante]{Marcos P. Cavalcante}
\author[Mendes]{Abraão Mendes}
\author[dos Santos]{Ian R. dos Santos}
\address{Institute of Mathematics, Federal University of Alagoas (UFAL),
Maceió--AL, Brazil}
\email{marcos.cavalcante@im.ufal.br}
\email{abraao.mendes@im.ufal.br}
\email{ian.santos@im.ufal.br}
\subjclass[2020]{Primary 53A10, 49Q05; Secondary 58E12, 30F15, 30F30}
\keywords{Free boundary minimal surfaces; Morse index; weighted harmonic one-forms; finite total curvature; conformal compactification; ends and multiplicities}
\date{\today}

\begin{abstract}
Let $\Om\subset\R^3$ be an unbounded domain with smooth boundary and let $\Sig$ be a complete, noncompact, orientable, immersed free boundary minimal surface in $\Om$, with compact boundary $\dSig\subset\dOm$ and finite Morse index. We prove that if the mean curvature of $\dOm$ satisfies $H_{\dOm}\ge0$ along $\dSig$ and $H_{\dOm}>0$ at some point of $\dSig$, then
\[
\Ind(\Sig)\ \ge\ \frac13\Bigl(2g+k+2\sum_{j=1}^r(d_j+1)-2\Bigr),
\]
where $g$ is the genus of $\Sig$, $k$ is the number of connected components of $\dSig$, $r$ is the number of ends of $\Sig$ and $d_1,\dots,d_r$ are their respective multiplicities. When $H_{\dOm}$ is only assumed to be nonnegative we obtain $\Ind(\Sig)\ge(2g+k-2)/3$, with the sharp bound $(2g+k-1)/3$ under a mild condition on the ends. The proofs use the harmonic one-form method of Ros and Chodosh--Máximo, with weighted $L^2$ spaces, adapted to the free boundary setting in the spirit of Ambrozio--Carlotto--Sharp. The main new ingredient is the computation of the dimension of the space of harmonic one-forms on a punctured compact Riemann surface with boundary which are tangential along the boundary and square integrable with respect to a weight. For a class of admissible weights $\rho$, we prove that this dimension is $2g+k-1+2\sum_jN_j-\varepsilon$, where $N_j$ is the maximal order of pole allowed by $\rho$ at the $j$-th puncture and $\varepsilon\in\{0,1\}$. For the weight of Chodosh--Máximo one has $N_j=d_j+1$.
\end{abstract}

\maketitle

\tableofcontents

\section{Introduction}\label{sec:intro}

The Morse index of a minimal surface measures the number of independent deformations that decrease area to second order, and its interplay with the topology of the surface has attracted considerable attention. A fundamental tool in this direction is the harmonic one-form method introduced by Ros \cite{Ros}, developed by Savo \cite{Savo} in the sphere, refined by Chodosh and Máximo \cite{CM1,CM2} for complete minimal surfaces of finite index in $\R^3$, and extended by Ambrozio, Carlotto and Sharp \cite{ACS1,ACS2} to broad classes of closed and of compact free boundary minimal hypersurfaces, where in dimension three the resulting estimates are expressed in terms of the genus and the number of boundary components. More recently, Hong and Saturnino \cite{HS} developed the spectral and structural theory of complete free boundary, and more generally capillary, surfaces of finite index. The purpose of the present paper is to establish index estimates for complete \emph{noncompact} free boundary minimal surfaces of finite Morse index with compact boundary, in terms of the genus, the number of boundary components, and, above all, the ends of the surface and their multiplicities, in the spirit of the second paper of Chodosh and Máximo \cite{CM2} in the boundaryless case.

\subsection{The geometric setting and the main results}

Throughout this paper, $\Om$ denotes an unbounded domain in $\R^3$ with smooth boundary $\dOm$, and $\Sig$ denotes a complete, noncompact, orientable, immersed free boundary minimal surface in $\Om$ with compact boundary $\dSig\subset\dOm$; free boundary means that $\Sig$ meets $\dOm$ orthogonally along $\dSig$. We write $H_{\dOm}$ for the mean curvature of $\dOm$ computed with respect to the outward unit normal of $\Om$.

Assume that $\Sig$ has finite Morse index. By the structure theorem of Hong and Saturnino \cite[Thm.~1.4]{HS}, recalled here as Theorem~\ref{thm:HS-structure}, the surface $\Sig$ is then conformally a compact Riemann surface with boundary $\Sgb$, of genus $g$ and with $k$ boundary components, punctured at $r$ interior points $p_1,\dots,p_r$,
\[
\Sig\ \cong\ \Sgb\setminus\{p_1,\dots,p_r\}, \qquad \dSgb=\dSig,
\]
the punctures corresponding to the ends $\Ed_1,\dots,\Ed_r$ of $\Sig$. Moreover, $\Sig$ has finite total curvature, so that each end is, with some multiplicity $d_j\ge1$, asymptotic to a plane or to a half-catenoid. One has $d_j=1$ exactly when the end is embedded, and $d_j\ge2$ for ends that wind onto a plane with higher multiplicity, as already happens for the Enneper end, where $d=3$. Our main result is the following.

\begin{theorem}\label{thm:main}
Let $\Sig$ be a complete, noncompact, orientable, immersed free boundary minimal surface of finite Morse index in $\Om$, with compact boundary $\dSig\subset\dOm$. Suppose that
\[
H_{\dOm}\ge0\ \text{ along }\ \dSig,\ \text{ and }\ H_{\dOm}(p)>0\ \text{ for some }\ p\in\dSig.
\]
Then
\begin{equation}\label{eq:main}
\Ind(\Sig)\ \ge\ \frac13\Bigl(2g+k+2\sum_{j=1}^r(d_j+1)-2\Bigr),
\end{equation}
where $g$ is the genus of $\Sig$, $k$ is the number of connected components of $\dSig$, $r$ is the number of ends of $\Sig$ and $d_1,\dots,d_r$ are their respective multiplicities.
\end{theorem}

The estimate \eqref{eq:main} is attained: the free boundary half-catenoid has
$g=0$, $k=1$ and a single embedded catenoidal end, so that the right-hand side
equals $1$. On the other hand, its Morse index is exactly $1$ \cite[Thm.~1.1]{CS2} (see also
\cite[Cor.~24]{ABCS3}).

Since $\Sig$ is noncompact, it has at least one end, and every multiplicity is at least~$1$. Hence the right-hand side of \eqref{eq:main} is at least $(2g+k+2)/3>0$, and in particular there is no stable surface in this class (see Corollary~\ref{cor:no-stable}). We point out that the harmonic one-form method with the usual $L^2$ test fields gives at most $(2g+k-1)/3$; see Theorem~\ref{thm:weak} below. The term $2\sum_{j=1}^r(d_j+1)$ in \eqref{eq:main} comes from the ends of $\Sig$, and it is obtained by working with a weighted $L^2$ space of harmonic fields, as in \cite{CM2}.

When $H_{\dOm}$ vanishes identically along $\dSig$, the boundary term in the second variation of area may vanish and the argument above does not work with the weighted space. In this case we use the $L^2$ space of harmonic fields, and we lose at most one unit in the estimate. More precisely, let
\[
d:=\dim_\R\bigl(\Lst(\Sig)\cap\Lsq(\Sig)\bigr),
\]
be the dimension of the space of rotational harmonic fields which are in $L^2$ and tangent to $\dSig$. We prove that $d\le1$ (Proposition~\ref{prop:dim-Lstar}).

\begin{theorem}\label{thm:weak}
Let $\Sig$ be as in Theorem~\ref{thm:main}, but assume only that $H_{\dOm}\ge0$ along $\dSig$. Then
\begin{equation}\label{eq:weak}
\Ind(\Sig)\ \ge\ \frac{2g+k-1-d}{3}\ \ge\ \frac{2g+k-2}{3}.
\end{equation}
In particular, $\Ind(\Sig)\ge(2g+k-1)/3$ whenever $d=0$, which happens for instance when $k=1$, when some end of $\Sig$ is catenoidal, or when the ends of $\Sig$ do not all share a common limiting normal.
\end{theorem}

\begin{remark}\label{rem:hypotheses}
The domain $\Om$ enters the proofs only through a neighborhood of $\dSig$: the second variation of area sees $\dOm$ only through its second fundamental form along the compact curve $\dSig$, and the compactification (Theorem~\ref{thm:HS-structure}) needs no hypothesis on $H_{\dOm}$ at all.
Theorems~\ref{thm:main} and~\ref{thm:weak} are therefore statements about a \emph{support surface}: a surface $S\subset\R^3$ containing $\dSig$, met orthogonally by $\Sig$ along $\dSig$ and only there, whose mean curvature, computed with respect to the unit normal pointing away from $\Sig$, is nonnegative along $\dSig$ (respectively positive somewhere on $\dSig$). 
Both the index of $\Sig$ and the hypotheses of the theorems depend on $S$ only along $\dSig$, and any unbounded domain $\Om$
with $\Sig\subset\overline\Om$, $\Sig\cap\dOm=\dSig$ and $\dOm=S$ near $\dSig$ may be used to phrase them.

The sign condition cannot be dropped. If $\Om$ is the exterior of $k\ge3$
disjoint round balls with collinear centers, so that $\dOm$ is mean-concave as
seen from $\Om$, a plane through the centers minus $k$ disks is a stable free
boundary minimal surface in $\Om$, and no estimate of the form \eqref{eq:main}
or \eqref{eq:weak} can hold.
\end{remark}

\begin{remark}\label{rem:bending}
The free boundary condition involves only the tangent
planes of $S$ along $\dSig$, that is, it is a first order condition on $S$.
On the other hand, the mean curvature of $S$ along $\dSig$ depends also on
the second derivatives of $S$ in the direction normal to $\dSig$ inside $S$.
In other words, we can bend $S$ near $\dSig$, by a quantity of order two in
the distance to $\dSig$, without moving $\Sig$ and without losing the
orthogonality. 

Let us describe the basic example. Let $\Sig$ be a complete noncompact free
boundary minimal surface in the half-space $\{x_3>0\}$ with compact boundary
in the plane $P=\{x_3=0\}$. Since $H_P\equiv0$, a priori only
Theorem~\ref{thm:weak} applies to $\Sig$. Now we bend $P$ near $\dSig$
towards the side of $\Sig$, in a parabolic way, and we obtain a new support
surface $S_\epsilon$ which is tangent to $P$ along $\dSig$ and has constant
mean curvature $2\epsilon>0$ along $\dSig$. If $\epsilon$ is small, then
$\Sig$ meets $S_\epsilon$ only along $\dSig$. 
Therefore, $\Sig$ satisfies the hypotheses of Theorem~\ref{thm:main} with respect to $S_\epsilon$.

We notice however that the index here is the index of $\Sig$ with respect to
$S_\epsilon$, and not with respect to $P$. Since the bending adds a
negative boundary term to the second variation formula, this index is
greater than or equal to the index of $\Sig$ with respect to $P$. Hence,
\eqref{eq:main} gives a lower bound for it as well.

Such surfaces arise for instance by halving symmetric ones. Let $M\subset\R^3$
be a complete orientable minimal surface of finite total curvature, invariant
under the reflection in a plane $P$, and suppose that no end of $M$ is
invariant under this reflection. Then $M\cap P$ is a finite union of $c\ge1$
compact geodesics of $M$, the reflection interchanges the ends in pairs, and
each half $\Sig$ of $M$ is a free boundary minimal surface with support
surface $P$, with $k=c$ boundary components, half of the ends of $M$ with
their multiplicities, and genus $h$ given by $\operatorname{genus}(M)=2h+c-1$.
There are many examples of such surfaces $M$; we refer to Berglund and
Rossman \cite{BerglundRossman}, Kapouleas \cite{Kapouleas} and Wohlgemuth
\cite{Wohlgemuth} for constructions of complete minimal surfaces of finite
total curvature with planes of symmetry.
\end{remark}

\subsection{The dimension formula}

Let us describe the strategy of the proof of Theorem~\ref{thm:main}. As in \cite{Ros,ACS2,CM1}, we use the coordinate functions of harmonic vector fields on $\Sig$ which are tangent to $\dSig$, multiplied by a cut-off function, as test functions for the second variation of area. If $V$ is a space of such fields for which the argument works, we obtain $\Ind(\Sig)\ge\dim V/3$. When $V$ is the space $\Lsq(\Sig)$ of $L^2$ harmonic fields tangent to $\dSig$, we have $\dim V=2g+k-1$ and we obtain Theorem~\ref{thm:weak}. Note that in this case the ends do not appear in the estimate, since an $L^2$ harmonic one-form extends smoothly across the punctures. In order to obtain Theorem~\ref{thm:main}, we follow \cite[\S3]{CM2} and we replace the $L^2$ condition by square integrability with respect to the weight
\begin{equation}\label{eq:weight-intro}
\rho(x)=\frac{1}{(1+|x|^2)\bigl(\log(2+|x|)\bigr)^2},\qquad x\in\R^3,
\end{equation}
restricted to $\Sig$. Since $\rho$ goes to zero along the ends, the corresponding space $\Lstar(\Sig)$ of weighted harmonic fields contains $\Lsq(\Sig)$ and also fields which blow up at the punctures. The main step of the proof is the computation of its dimension.

We do this computation in Sections~\ref{sec:conformal} and \ref{sec:weighted} in a more general setting, since it depends only on the conformal structure of $\Sig$ and on the behavior of the weight near the punctures. More precisely, we consider a compact Riemann surface with boundary punctured at finitely many interior points, and a weight $\rho$ which near each puncture $p_j$ is comparable to a radial function. To such a weight we associate an integer $N_j\ge0$, namely the largest order of pole at $p_j$ with finite energy with respect to $\rho$ (see Definition~\ref{def:admissible}). Our second result is the following.

\begin{theorem}\label{thm:dim-formula}
Let $\Sgb$ be a compact Riemann surface of genus $g$ with $k\ge1$ boundary components, let $p_1,\dots,p_r$ be distinct interior points and set $\Sig=\Sgb\setminus\{p_1,\dots,p_r\}$. Let $\rho$ be an admissible weight with local orders $N_1,\dots,N_r$. Then the space $\Lrho(\Sig)$ of harmonic one-forms on $\Sig$ that are tangential along $\dSig$ and satisfy
\[
\int_\Sig|\omega|^2\rho\dd\mu<\infty
\]
has real dimension
\[
\dim_\R\Lrho(\Sig)=2g+k-1+2\sum_{j=1}^rN_j-\varepsilon,
\]
where
\[
\varepsilon=
\begin{cases}
1,&\text{if }\ \sum_{j}N_j\ge1,\\[4pt]
0,&\text{if }\ N_1=\cdots=N_r=0.
\end{cases}
\]
\end{theorem}

Here $2g+k-1$ is the dimension of the space of harmonic one-forms on $\Sgb$ which are tangential along $\dSgb$, the term $2N_j$ counts the real and imaginary parts of the coefficients of the principal part at $p_j$, and $\varepsilon$ comes from a linear relation among the residues, which follows from Stokes' theorem (Lemma~\ref{lem:residue-necessity}). The proof is elementary. After computing the Laurent expansion at each puncture (Lemma~\ref{lem:laurent}), we show that this relation is the only obstruction to prescribing the principal parts, by solving a Dirichlet and a Neumann problem on $\Sgb$ (Lemma~\ref{lem:sufficiency}). In Remark~\ref{rem:double} we observe that the same formula can be obtained from the Riemann--Roch theorem on the double of $\Sgb$.

In Section~\ref{sec:fbms} we go back to minimal surfaces. Using the asymptotic behavior $|X(z)|\asymp|z|^{-d_j}$ of an end of multiplicity $d_j$, we show that the weight \eqref{eq:weight-intro} is comparable near $p_j$ to $r^{2d_j}(\log(1/r))^{-2}$, and that $N_j=d_j+1$ (Lemma~\ref{lem:N-computation}). Since $N_j\ge2$, we have $\varepsilon=1$, and Theorem~\ref{thm:dim-formula} gives
\begin{equation}\label{eq:dim-intro}
\dim_\R\Lstar(\Sig)=2g+k+2\sum_{j=1}^r(d_j+1)-2,
\end{equation}
which is the number appearing in \eqref{eq:main}.

\subsection{The logarithmic cut-off}

The identity we use to compute the second variation is
\begin{equation}\label{eq:identity-intro}
\cQ(\varphi\xi,\varphi\xi)=\int_\Sig|\nabla\varphi|^2|\xi|^2\dd\mu-\int_{\dSig}H_{\dOm}|\xi|^2\dd\sigma,
\end{equation}
where $\xi$ is a harmonic field tangent to $\dSig$ and $\varphi$ is a cut-off function with $\varphi\equiv1$ near $\dSig$ (Lemma~\ref{lem:key-identity}). In order to conclude, we need the first term on the right-hand side to be small. If $\xi\in\Lsq(\Sig)$, this follows from the usual cut-off $\varphi_R$, with $|\nabla\varphi_R|\le C/R$. However, a field $\xi\in\Lstar(\Sig)$ is not in $L^2$ in general, and this estimate is not enough. Following \cite[\S4]{CM2}, we use instead a logarithmic cut-off $\lcut_R$, which satisfies $|\nabla\lcut_R|^2\le C_0\mskip1mu\rho$ with $C_0$ independent of $R$ (see \S\ref{subsec:cutoff}). Then the first term is bounded by $C_0\int_{\Sig\cap\{|x|\ge R\}}|\xi|^2\rho\dd\mu$, which goes to zero as $R\to+\infty$. In Remark~\ref{rem:optimality} we explain why a weight decaying faster than $\rho$ would not improve the estimate.

\subsection{An open question}\label{subsec:open-question}

Theorem~\ref{thm:weak} is proved with the space $\Lsq(\Sig)$, and so the ends do not contribute to the estimate. We do not know whether an estimate like \eqref{eq:main}, involving the ends, holds when we only assume $H_{\dOm}\ge0$ along $\dSig$. The proof of the bound $d\le1$ in Proposition~\ref{prop:dim-Lstar} uses that the field $\star\nabla\langle\vec a,X\rangle$ is not in $L^2(\Sig)$ unless every end has limiting normal $\pm\vec a/|\vec a|$. This argument does not work for the weighted space, since $\int_\Sig\rho\dd\mu<+\infty$ (Lemma~\ref{lem:rho-integrable}) and hence every bounded harmonic field, in particular every $\star\mskip1mu\xi_i$, has finite $\rho$-energy. We were not able to find a replacement for this argument.

\subsection{Organization}

Section~\ref{sec:prelim} fixes the geometric setting, the weight and the two cut-offs. Sections~\ref{sec:conformal} and \ref{sec:weighted} develop the conformal framework and prove Theorem~\ref{thm:dim-formula}. Section~\ref{sec:fbms} computes the local orders of \eqref{eq:weight-intro} and deduces \eqref{eq:dim-intro}. Section~\ref{sec:boundary} proves the identity \eqref{eq:identity-intro}, the strict positivity of the boundary term, and the bound $d\le1$. Section~\ref{sec:proof} proves Theorems~\ref{thm:main} and \ref{thm:weak}.

\section{Preliminaries on free boundary minimal surfaces}\label{sec:prelim}

Throughout this section, $\Om\subset\R^3$ is an unbounded domain with smooth boundary $\dOm$, and $\Sig$ is a connected, complete, noncompact, orientable, immersed free boundary minimal surface in $\Om$ with compact boundary $\dSig\subset\dOm$. We may assume that $\Sig$ is connected, since the index of a disjoint union is the sum of the indices of its components. We fix a unit normal vector field $N$ along $\Sig$ and denote by $\nu$ the outward unit conormal of $\dSig$ in $\Sig$. We write $D$ for the Levi-Civita connection of $\R^3$ and $\nabla$ for the induced connection on $\Sig$, and we denote by $X=(x_1,x_2,x_3)$ the position vector of the immersion. 

\subsection{Second variation and Morse index}\label{subsec:secondvar}

The Weingarten operator of $\Sig$ is $A(Y)=D_YN$, $Y\in\X(\Sig)$, and minimality of $\Sig$ means $\tr A=0$. The second fundamental form of $\dOm$ is 
\[
\II_{\dOm}(Y,Z)=\langle D_Y\eta,Z\rangle,\qquad Y,Z\in\X(\dOm),
\]
where $\eta$ is the outward unit normal of $\dOm$, and $H_{\dOm}=\tr\II_{\dOm}$ is the mean curvature of $\dOm$. The free boundary condition $T_x\Sig\perp T_x(\dOm)$ for $x\in\dSig$ is equivalent to $\nu=\eta$ along $\dSig$; in particular, $N(x)\in T_x(\dOm)$ for every $x\in\dSig$.

For an admissible variation with normal variational field $\phi N$, with $\phi\in C_c^\infty(\Sig)$, the second variation of area is the quadratic form
\begin{equation}\label{eq:Q}
\begin{aligned}
\cQ(\phi,\phi)&=\int_\Sig\bigl(|\nabla\phi|^2-|A|^2\phi^2\bigr)\dd\mu-\int_{\dSig}\II_{\dOm}(N,N)\mskip1mu\phi^2\dd\sigma\\[4pt]
&=-\int_\Sig\phi\mskip2mu\cJ\phi\dd\mu+\int_{\dSig}\phi\mskip2mu\cB\mskip1mu\phi\dd\sigma,
\end{aligned}
\end{equation}
where
\[
\cJ=\Delta+|A|^2\quad\text{and}\quad\cB\mskip1mu\phi=\p_\nu\mskip1mu\phi-\II_{\dOm}(N,N)\mskip1mu\phi
\]
are the Jacobi operator and the associated Robin-type boundary operator along $\dSig$; see \cite{ACS2,HS}. Functions $\phi$ need not vanish on $\dSig$. The Morse index $\Ind(\Sig)$ is the supremum of the dimensions of the subspaces of $C_c^\infty(\Sig)$ on which $\cQ$ is negative definite; $\Sig$ is \emph{stable} if $\Ind(\Sig)=0$, and we assume throughout that $\Ind(\Sig)<\infty$. Given smooth maps $Y=(y_1,y_2,y_3)$ and $Z=(z_1,z_2,z_3)$ from $\Sig$ to $\R^3$, one of compact support, we extend $\cQ$ to $\R^3$-valued maps by
\begin{equation}\label{eq:Q-vec}
\cQ(Y,Z):=\sum_{i=1}^3\cQ(y_i,z_i).
\end{equation}

\subsection{Conformal compactification and the ends}\label{subsec:conformal}

We first recall the structure theorem of Hong and Saturnino in the form in which we shall use it. Their result is stated for capillary surfaces of finite index in Riemannian $3$-manifolds, under the hypothesis $R_M+H_\Sig^2\ge0$ together with one of two alternatives, the first of which is that $\dSig$ be compact; in our situation this alternative holds by hypothesis, so no assumption on $H_{\dOm}$ is needed here.

\begin{theorem}[Hong--Saturnino, {\cite[Thm.~1.4]{HS}}]\label{thm:HS-structure}
Let $\Sig$ be a complete noncompact\linebreak free boundary minimal surface of finite Morse index immersed in $\Om\subset\R^3$, with compact boundary $\dSig\subset\dOm$. Then $\Sig$ is conformally equivalent to a compact\linebreak Riemann surface with boundary, punctured at finitely many points, each associated to an end of $\Sig$, and
\[
\int_\Sig|A|^2\dd\mu+\int_{\dSig}H_{\dOm}\dd\sigma<+\infty.
\]
\end{theorem}

Write $\Sgb$ for the compactification and $p_1,\dots,p_r$ for the punctures. Since $\dSig$ is compact, no puncture can lie on $\dSgb$, so all punctures are interior points:
\begin{equation}\label{eq:model}
\Sig\ \cong\ \Sgb\setminus\{p_1,\dots,p_r\}, \qquad p_1,\dots,p_r\in\Sgb\setminus\dSgb, \qquad \dSgb=\dSig.
\end{equation}
In particular, $\Sig$ has finite genus $g$, $k$ boundary components, $r$ ends $\Ed_1,\dots,\Ed_r$, and finite total curvature. By the work of Jorge and Meeks \cite{JorgeMeeks} and Schoen \cite{Schoen}, each end is then asymptotic to a plane or to a half-catenoid, the unit normal $N$ converges along each end to a constant unit vector, and $\Sig$ is properly immersed; in particular, $\Sig\cap\overline{B}_s$ is compact for every $s>0$, where $B_s$ is the open Euclidean ball of radius $s$ centered at the origin.

We shall use the notion of multiplicity of an end and the associated asymptotics in the form given by Chodosh and Máximo \cite[\S2]{CM2}. Fix a holomorphic coordinate $z$ centered at $p_j$, defined on a disk $D_j$ with $D_j^\ast:=D_j\setminus\{p_j\}\subset\Sig$, and write $z=re^{i\theta}$. Then there is an integer $d_j\ge1$, the \emph{multiplicity} of the end $\Ed_j$, such that
\begin{equation}\label{eq:dist-end}
|X(z)|\ \asymp\ |z|^{-d_j} \quad \text{as} \quad z\to0,
\end{equation}
and such that the induced metric, written near $p_j$ in the conformal coordinate, takes the form $\lambda_j(z)^2|\mathrm{d}z|^2$ with
\begin{equation}\label{eq:conf-factor}
\lambda_j(z)\ \asymp\ |z|^{-(d_j+1)} \quad \text{as} \quad z\to0.
\end{equation}
Here and below, $a\asymp b$ means that the quotient of the two quantities is bounded above and below by positive constants, while $a\sim b$ means that it tends to $1$. The end is embedded exactly when $d_j=1$. We use \eqref{eq:dist-end} only through the two consequences
\begin{equation}\label{eq:end-consequences}
1+|X(z)|^2\ \asymp\ r^{-2d_j} \quad \text{and} \quad \log\bigl(2+|X(z)|\bigr)\ \sim\ d_j\log(1/r).
\end{equation}

\subsection{Harmonic vector fields and harmonic one-forms}\label{subsec:fields}

We denote by $\HH^1(\Sig)$ the space of \emph{harmonic tangent vector fields} on $\Sig$, that is, of $\xi\in\X(\Sig)$ satisfying
\begin{equation}\label{eq:harmonic-field}
\begin{cases}
\diver\xi=0,\\[2pt]
\langle\nabla_Y\xi,Z\rangle=\langle\nabla_Z\xi,Y\rangle,\ \text{ for all }\ Y,Z\in\X(\Sig),
\end{cases}
\end{equation}
by $\HH^1_T(\Sig)$ the subspace of fields that are tangent to the boundary, $\langle\xi,\nu\rangle=0$ along $\dSig$, and we set
\[
\Lsq(\Sig)=\Bigl\{\xi\in\HH^1_T(\Sig):\textstyle\int_\Sig|\xi|^2\dd\mu<+\infty\Bigr\}.
\]

To a tangent vector field $\xi$ we associate the one-form $\omega=\xi^\flat$. The two conditions in \eqref{eq:harmonic-field} read $d^\ast\omega=0$ and $d\mskip1mu\omega=0$, so $\xi\in\HH^1(\Sig)$ if and only if $\omega$ is a harmonic one-form; moreover, $\langle\xi,\nu\rangle=\iota_\nu\omega$, so $\xi$ is tangent to $\dSig$ if and only if $\omega$ is tangential. The pointwise norms agree, $|\xi|=|\omega|$. We pass freely between the two languages, and we use the same symbols $\HH^1(\Sig)$, $\HH^1_T(\Sig)$, $\Lsq(\Sig)$ for either space; the systematic treatment of harmonic one-forms is deferred to Section~\ref{sec:conformal}.

Let $\{E_1,E_2,E_3\}$ be the standard basis of $\R^3$ and let $\xi_i=E_i^\top=\nabla x_i$ be the tangential projection of $E_i$ onto $\Sig$. Since $\Sig$ is minimal, each $\xi_i$ is harmonic, and so is each $\star\mskip1mu\xi_i$, where $\star$ denotes the Hodge star operator; we set
\[
\Lst(\Sig)=\spann\{\star\mskip1mu\xi_1,\star\mskip1mu\xi_2,\star\mskip1mu\xi_3\}.
\]
For a tangent field $\xi$, we write throughout
\begin{equation}\label{eq:coordinates}
u_i=\langle\xi,E_i\rangle,\qquad g_i=\langle N,E_i\rangle\qquad(i=1,2,3),
\end{equation}
so that $\sum_iu_i^2=|\xi|^2$, $\sum_ig_i^2=1$, $N=(g_1,g_2,g_3)$ and, since $\xi$ is tangent to $\Sig$, $\sum_iu_ig_i=\langle\xi,N\rangle=0$. Note that $\sum_i|\xi_i|^2=2$ pointwise, so that the three fields $\xi_i$ cannot all belong to $L^2(\Sigma)$; equivalently, neither $\{\xi_1,\xi_2,\xi_3\}$ nor $\{\star\xi_1,\star\xi_2,\star\xi_3\}$ is contained in $L^2(\Sigma)$.

\subsection{The weight and the two cut-off functions}\label{subsec:cutoff}

Following \cite[\S3]{CM2}, we fix once and for all the weight
\begin{equation}\label{eq:rho-def}
\rho(x):=\frac{1}{(1+|x|^2)\bigl(\log(2+|x|)\bigr)^2},\qquad x\in\R^3,
\end{equation}
and we regard it, by restriction, as a positive smooth function on $\Sig$. The space
\[
\Lstar(\Sig):=\Bigl\{\xi\in\HH^1_T(\Sig):\textstyle\int_\Sig|\xi|^2\rho\dd\mu<+\infty\Bigr\}
\]
of \emph{weighted tangential harmonic fields} contains $\Lsq(\Sig)$, since $\rho\le\rho(0)\approx2.08$, and Section~\ref{sec:fbms} will show that the containment is strict and compute the deficiency. The following property of the weight will be used in \S\ref{subsec:open-question}.

\begin{lemma}\label{lem:rho-integrable}
The weight \eqref{eq:rho-def} is integrable on $\Sig$:
\[
\int_\Sig\rho\dd\mu<+\infty.
\]
Consequently, every bounded harmonic field on $\Sig$, and in particular each of $\xi_1,\xi_2,\xi_3$ and $\star\mskip1mu\xi_1,\star\mskip1mu\xi_2,\star\mskip1mu\xi_3$, has finite $\rho$-energy.
\end{lemma}

\begin{proof}
It suffices to bound the integral over each punctured coordinate disk $D_j^\ast$, the complement being compact. By \eqref{eq:end-consequences},
\begin{equation}\label{eq:rho-profile}
\rho\ \asymp\ r^{2d_j}\bigl(\log(1/r)\bigr)^{-2} \quad \text{as} \quad r\to0^+,
\end{equation}
while $\dd\mu=\lambda_j^2\dd x\dd y\asymp r^{-2(d_j+1)}\,r\dd r\dd\theta$ by \eqref{eq:conf-factor}. Hence
\[
\int_{D_j^\ast}\rho\dd\mu\ \asymp\ 2\pi\int_0^{R_1}\frac{\dd r}{r\bigl(\log(1/r)\bigr)^{2}},
\]
which the substitution $t=\log(1/r)$ identifies with a convergent integral of $t^{-2}$. The final assertion is immediate, since $|\xi_i|\le1$ and $\left|\star\mskip1mu\xi_i\right|=|\xi_i|$.
\end{proof}

We now describe the two cut-off constructions. Fix $R_0\ge2$ with $\dSig\subset B_{R_0}$, which is possible because $\dSig$ is compact. Recall from \S\ref{subsec:conformal} that $\Sig\cap\overline{B}_s$ is compact for every $s>0$.

\subsubsection*{The linear cut-off}

Adapting \cite[\S2]{CM1}, for every $R\ge R_0$ there is $\varphi_R\in C_c^\infty(\Sig)$ with $0\le\varphi_R\le1$, $\varphi_R\equiv1$ on $\Sig\cap B_R$, $\supp\varphi_R\subset\Sig\cap B_{2R}$, and
\[
\text{(iii)}\ \ |\nabla\varphi_R|\le\frac CR\ \text{ on }\ \Sig,\qquad \text{(iv)}\ \ |\varphi_R\Delta\varphi_R|\le\frac C{R^2}\ \text{ on }\ \Sig\cap(B_{2R}\setminus B_R),
\]
with $C>0$ independent of $R$; in particular, $\varphi_R\equiv1$ near $\dSig$ and $\varphi_R\to1$ pointwise. This cut-off is adequate for the unweighted theory: if $\xi\in\Lsq(\Sig)$ then (iii) gives
\begin{equation}\label{eq:linear-gradient}
\int_\Sig|\nabla\varphi_R|^2|\xi|^2\dd\mu\ \le\ \frac{C^2}{R^2}\mskip2mu\|\xi\|_{L^2(\Sig)}^2\ \longrightarrow\ 0 \qquad(R\to+\infty).
\end{equation}

\subsubsection*{The logarithmic cut-off}

For the weighted theory, the estimate \eqref{eq:linear-gradient} is unavailable: a field $\xi\in\Lstar(\Sig)$ need not be square integrable, and the annulus supporting $\nabla\varphi_R$ moves, as $R\to+\infty$, towards the poles that the weight admits. We therefore replace $\varphi_R$ by the logarithmic cut-off of Chodosh and Máximo \cite[\S4]{CM2}, calibrated to the weight \eqref{eq:rho-def}.

Fix $\chi\in C^\infty(\R)$ with
\[
\chi\equiv0\ \text{ on }\ (-\infty,0],\qquad\chi\equiv1\ \text{ on }\ [1,+\infty),\qquad 0\le\chi\le1,
\]
and set $c_\chi:=\sup_\R|\chi'|$, a constant depending only on the choice of $\chi$. For $R\ge R_0$, define
\[
\lcut_R:=\chi\circ\psi_R\Big|_\Sig,\quad\text{where}\quad\psi_R(x):=2-\frac{\log|x|}{\log R}\quad (x\ne0).
\]
We list the properties of $\lcut_R$ that will be used. If $0<|x|\le R$ then $\psi_R(x)\ge1$, and if $|x|\ge R^2$ then $\psi_R(x)\le0$; hence
\[
0\le\lcut_R\le1,\qquad\lcut_R\equiv1\ \text{ on }\ \Sig\cap B_R,\qquad
\supp\lcut_R\subset\Sig\cap\overline{B}_{R^2},
\]
so that $\lcut_R\in C_c^\infty(\Sig)$, $\lcut_R\equiv1$ on a neighborhood of $\dSig$, and $\lcut_R\to1$ pointwise as $R\to+\infty$. Since the gradient of $|x|$ along $\Sig$ is the tangential projection of a unit vector,
$\bigl|\nabla|x|\bigr|\le1$ and
\begin{equation}\label{eq:gradcutoff}
|\nabla\lcut_R|=|\chi'\circ\psi_R|\mskip2mu|\nabla\psi_R|\ \le\ \frac{c_\chi}{|x|\log R},
\end{equation}
while $\nabla\lcut_R$ vanishes outside the transition region $\mathcal A_R:=\Sig\cap(B_{R^2}\setminus B_R)$, where $\lcut_R$ is locally constant. The main property of $\lcut_R$ is the following pointwise comparison with the weight:
\begin{equation}\label{eq:cutoff-vs-rho}
|\nabla\lcut_R|^2\ \le\ C_0\mskip1mu\rho\ \text{ on }\ \Sig,\ \text{ with }\ C_0:=18\mskip1mu c_\chi^2\ \text{ independent of }\ R.
\end{equation}
Indeed, the left-hand side of \eqref{eq:cutoff-vs-rho} vanishes off $\mathcal A_R$, while on $\mathcal A_R$, assuming that $R\ge2$, one has $R\le|x|\le R^2$, whence $\log(2+|x|)\le\log(2R^2)\le3\log R$ and $1+|x|^2\le2|x|^2$; therefore, by \eqref{eq:gradcutoff} and the definition \eqref{eq:rho-def} of $\rho$,
\[
|\nabla\lcut_R(x)|^2\ \le\ \frac{c_\chi^2}{|x|^2(\log R)^2}\ \le\ c_\chi^2\mskip2mu\frac{2}{1+|x|^2}\cdot\frac{9}{\bigl(\log(2+|x|)\bigr)^2}\ = \ 18\mskip2mu c_\chi^2\mskip1mu\rho\mskip1mu(x).
\]

\smallskip
From \eqref{eq:cutoff-vs-rho} we obtain the following lemma, which is how the logarithmic cut-off will be used. In the proof, we use that $\Lstar(\Sig)$ is finite-dimensional (Corollary~\ref{cor:fbms}).

\begin{lemma}\label{lem:cutoff-to-zero}
Let $\xi\in\Lstar(\Sig)$. Then
\begin{equation}\label{eq:cutoff-to-zero-1}
\int_\Sig|\nabla\lcut_R|^2\mskip1mu|\xi|^2\dd\mu\ \le\ C_0\int_{\Sig\cap\{|x|\ge R\}}|\xi|^2\rho\dd\mu\ \longrightarrow\ 0\qquad(R\to+\infty).
\end{equation}
Moreover, there is a sequence $\epsilon_R\searrow0$, depending on $\Lstar(\Sig)$ but not on the individual field, such that
\begin{equation}\label{eq:cutoff-to-zero-2}
\int_\Sig|\nabla\lcut_R|^2\mskip1mu|\xi|^2\dd\mu\ \le\ C_0\mskip2mu\epsilon_R\mskip2mu\|\xi\|_\rho^2\ \text{ for all } R\ge R_0 \text{ and all } \xi\in\Lstar(\Sig),
\end{equation}
where $\|\xi\|_\rho^2:=\int_\Sig|\xi|^2\rho\dd\mu$.
\end{lemma}

\begin{proof}
Since $\nabla\lcut_R$ is supported in $\mathcal A_R\subset\Sig\cap\{|x|\ge R\}$, the comparison \eqref{eq:cutoff-vs-rho} gives the inequality in \eqref{eq:cutoff-to-zero-1}, whose right-hand side is the tail of a convergent integral. For the uniform statement, using that $\Lstar(\Sig)$ is finite-dimensional, fix a $\|\mskip1mu\cdot\mskip1mu\|_\rho$-orthonormal basis $\eta_1,\dots,\eta_m$ of $\Lstar(\Sig)$; writing $\xi=\sum_lc_l\eta_l$, the pointwise Cauchy--Schwarz inequality gives $|\xi|^2\le\|\xi\|_\rho^2\sum_l|\eta_l|^2$, whence \eqref{eq:cutoff-to-zero-2} holds with $\epsilon_R:=\int_{\Sig\cap\{|x|\ge R\}}\bigl(\sum_l|\eta_l|^2\bigr)\rho\dd\mu\searrow0$, the tail of the convergent integral 
\[
\int_\Sig\Big(\sum_l|\eta_l|^2\Big)\rho\dd\mu=\sum_l\|\eta_l\|_\rho^2=m.\qedhere
\]
\end{proof}

\subsection{Integrability estimates}\label{subsec:integrability}

The following facts will be used in the proof of Theorem~\ref{thm:weak}. They follow
from finite total curvature through the linear cut-off; compare
\cite[\S4]{CM2}.

\begin{lemma}\label{lem:aux}
Suppose $\Sig$ has finite Morse index. Then:
\begin{enumerate}[label=\rm(\alph*),itemsep=3pt,topsep=4pt]
\item $|\nabla\xi|\in L^2(\Sig)$ for every $\xi\in\Lsq(\Sig)$;
\item if $\phi:\Sig\to\R$ is smooth, square integrable, and satisfies $\cJ\phi+\lambda\phi=0$ on $\Sig$ for some $\lambda\in\R$, then $|\nabla\phi|\in L^2(\Sig)$;
\item $|\nabla(\phi\langle\xi,\vec a\rangle)|\in L^1(\Sig)$ for every $\xi$ and $\phi$ as in {\rm(a)} and {\rm(b)} and every $\vec a\in\R^3$.
\end{enumerate}
\end{lemma}

\begin{proof}
Throughout, $\varphi_R$ is the linear cut-off, so that $\varphi_R\equiv1$ near $\dSig$ and $\p_\nu\varphi_R=0$ along $\dSig$.

\emph{\rm(a)} Green's identity for the pair $\varphi_R^2$, $|\xi|^2$, the Bochner formula 
\[
\Delta|\xi|^2=2|\nabla\xi|^2+2K|\xi|^2
\]
for the harmonic form $\omega=\xi^\flat$, and $\Delta\varphi_R^2=2\varphi_R\Delta\varphi_R+2|\nabla\varphi_R|^2$ combine into
\[
\int_\Sig\varphi_R^2|\nabla\xi|^2\dd\mu=\int_\Sig|\xi|^2\bigl(\varphi_R\Delta\varphi_R+|\nabla\varphi_R|^2-\varphi_R^2K\bigr)\dd\mu+\int_{\dSig}\langle\xi,D_\nu\xi\rangle\dd\sigma.
\]
By properties (iii) and (iv) of the linear cut-off, and since finite total curvature bounds $|A|$, hence $-K=|A|^2/2$, uniformly, the first integrand is bounded by $3C_1|\xi|^2$; the boundary term is finite and independent of $R$, $\dSig$ being compact and $\xi$ smooth up to it. Letting $R\to+\infty$ proves (a).

\emph{\rm (b)} Integrating $\diver(\varphi_R^2\phi\nabla\phi)$ over $\Sig$, using $\Delta\phi=-(|A|^2+\lambda)\phi$ and absorbing the cross term $-2\varphi_R\phi\langle\nabla\varphi_R,\nabla\phi\rangle$ by Young's inequality, one obtains
\[
\frac12\int_\Sig\varphi_R^2|\nabla\phi|^2\dd\mu\ \le\ \int_{\dSig}|\phi\mskip2mu\p_\nu\phi|\dd\sigma+\int_\Sig\bigl(|A|^2+|\lambda|+2|\nabla\varphi_R|^2\bigr)\phi^2\dd\mu,
\]
and (b) follows upon letting $R\to+\infty$, the right-hand side being bounded uniformly in $R$.

\emph{\rm (c)} Using \eqref{eq:harmonic-field}, one computes
\[
\nabla\langle\vec a,\xi\rangle=(\nabla\xi)\bigl(\vec a^\top\bigr)-\langle\vec a,N\rangle A(\xi),
\]
whence $\nabla\langle\vec a,\xi\rangle\in L^2(\Sig)$, since $|\nabla\xi|$ lies in $L^2(\Sig)$ by (a) and $|A|$ is bounded. The Leibniz rule and Hölder's inequality then give 
\[
\bigl|\nabla(\phi\langle\vec a,\xi\rangle)\bigr|\le|\phi|\mskip2mu|\nabla\langle\vec a,\xi\rangle|+|\vec a|\mskip2mu|\xi|\mskip2mu|\nabla\phi|\in L^1(\Sig),
\]
all four factors lying in $L^2(\Sig)$ by (a) and (b).
\end{proof}

\subsection{Auxiliary results}\label{subsec:aux}

A key ingredient is the following result of Ros \cite[\S2]{Ros}.

\begin{theorem}[Ros]\label{thm:Ros}
Let $\xi\in\HH^1(\Sig)$. Then
\[
\cJ\langle\xi,\vec a\rangle=-2\langle\nabla\xi,A\rangle\langle N,\vec a\rangle
\]
for every parallel vector field $\vec a\in\R^3$. Moreover, if $\Sig$ is not totally geodesic, then $\langle\nabla\xi,A\rangle\equiv0$ if and only if $\xi\in\Lst(\Sig)$.
\end{theorem}

Applying Theorem~\ref{thm:Ros} with $\vec a=E_i$ and $\cJ=\Delta+|A|^2$, the coordinate functions \eqref{eq:coordinates} of a harmonic field $\xi\in\HH^1(\Sig)$ satisfy
\begin{equation}\label{eq:Jui}
\Delta u_i=-|A|^2u_i-2\mskip1mu g_i\langle\nabla\xi,A\rangle,\qquad i=1,2,3;
\end{equation}
equivalently, $\cJ\xi=-2\langle\nabla\xi,A\rangle N$, so that $\langle\xi,\cJ\xi\rangle=0$, since $\xi$ is tangent to $\Sig$.

The next lemma is the unique continuation principle at the free boundary.\linebreak It concerns harmonic forms with vanishing Cauchy data on a boundary arc; see \cite[Thm.~3.4.4]{Schwarz}, and also \cite[Lem.~2.1]{HS}.

\begin{lemma}\label{lem:unique-cont}
Let $\omega$ be a harmonic one-form on $\Sig$, tangential along $\dSig$ and smooth up to $\dSig$. If $\omega\equiv0$ on a nonempty open arc $U\subset\dSig$, then $\omega\equiv0$ on $\Sig$.
\end{lemma}

The statement is local near the compact curve $\dSig$ and makes no reference to the behavior of $\omega$ at infinity; in particular, it applies to every $\omega\in\Lstar(\Sig)$ as well as to every $\omega\in\Lsq(\Sig)$.

Finally, we record the $L^2$ spectral characterization of the index due to Hong and Saturnino \cite[Prop.~3.13]{HS}, which plays in the free boundary setting the role of the theorem of Fischer-Colbrie \cite[Prop.~2]{FC} in the boundaryless one.

\begin{theorem}[Hong--Saturnino]\label{thm:HS}
Suppose that $\Sig$ has finite Morse index and put $n=\Ind(\Sig)$. Then there is an $n$-dimensional subspace $W\subset L^2(\Sig)$ admitting an orthonormal basis consisting of smooth functions $\phi_1,\dots,\phi_n$ such that
\[
\begin{cases}
\cJ\phi_i+\lambda_i\phi_i=0 & \text{on }\Sig,\\
\cB\mskip1mu\phi_i=0 & \text{along }\dSig,
\end{cases}
\]
with $\lambda_1,\dots,\lambda_n<0$, and such that
\[
\cQ(f,f)\ge0\ \text{ for every }\ f\in C_c^\infty(\Sig)\cap W^\perp.
\]
\end{theorem}

\section{Harmonic one-forms on a punctured bordered Riemann surface}\label{sec:conformal}

In this section and in the next one we do not use minimal surfaces. We consider a compact Riemann surface with boundary, punctured at finitely many interior points, together with a weight, and we compute the dimension of the space of tangential harmonic one-forms which are square integrable with respect to the weight.

\subsection{The setting}\label{subsec:setting}

Throughout Sections~\ref{sec:conformal} and \ref{sec:weighted}, let $\Sgb$ be a compact Riemann surface with nonempty boundary $\dSgb$, of genus $g$ and with $k\ge1$ boundary components. Let $p_1,\dots,p_r$ be distinct points of $\Sgb\setminus\dSgb$, which we call the \emph{punctures}, and set
\begin{equation}\label{eq:setting}
\Sig:=\Sgb\setminus\{p_1,\dots,p_r\}.
\end{equation}
Thus $\Sig$ is a surface with compact boundary $\dSig=\dSgb$, noncompact as soon as $r\ge1$. We fix once and for all a smooth Riemannian metric $h$ on $\Sgb$ compatible with the conformal structure; by Lemma~\ref{lem:conf-inv} below, nothing in what follows depends on this choice. For each $j$, we fix a holomorphic coordinate
\[
z: D_j\to\C,\qquad z(p_j)=0,
\]
defined on a coordinate disk $D_j\subset\Sgb\setminus\dSgb$, with the $D_j$ pairwise disjoint, and we write $D_j^\ast=D_j\setminus\{p_j\}\subset\Sig$, $z=re^{i\theta}$, and $R_1>0$ for a radius with $\{|z|\le R_1\}\subset D_j$ for every $j$. By \eqref{eq:model}, a complete noncompact free boundary minimal surface of finite Morse index with compact boundary is exactly of this form, and the notation is consistent with that of Section~\ref{sec:prelim}.

\subsection{Harmonic forms, tangentiality, and complexification}\label{subsec:harmonic}

On one-forms on an oriented Riemannian surface, the Hodge star operator $\star$ is the pointwise rotation by $+\pi/2$: in a positively oriented orthonormal coframe $\{e^1,e^2\}$,
\begin{equation}\label{eq:star}
\star\mskip1mu e^1=e^2,\qquad \star\mskip1mu e^2=-e^1,\qquad \star\mskip1mu\star=-\mathrm{id}.
\end{equation}
On functions and two-forms, $\star\mskip1mu1=\dd\mu$ and $\star\dd\mu=1$; in particular, $\star\mskip1mu\star=\mathrm{id}$ in either case. The codifferential on one-forms is $d^\ast=-\star d\mskip1mu\star$, and a one-form $\omega$ is \emph{harmonic} if
\[
d\mskip1mu\omega=0\quad\text{and}\quad d^\ast\omega=0.
\]
Since $\star\mskip1mu\star=\mathrm{id}$ on two-forms on surfaces, the condition $d^\ast\omega=0$ is equivalent to
\begin{equation}\label{eq:star-omega-closed}
d(\star\mskip1mu\omega)=0,
\end{equation}
which will be used in \S\ref{subsec:necessity}.

Along $\dSig$, let $\tau$ be the positively oriented unit tangent and $\nu$ the outward unit conormal. We orient $\Sig$ so that $\{\nu,\tau\}$ is a positively oriented orthonormal frame along the boundary, which is the convention that makes Stokes' theorem hold in the usual form. A one-form is \emph{tangential} if $\iota_\nu\omega=0$ and \emph{normal} if $\iota_\tau\omega=0$. Taking $e^1=\nu^\flat$ and $e^2=\tau^\flat$ in \eqref{eq:star} gives
\begin{equation}\label{eq:star-frame}
\star\mskip1mu\nu^\flat=\tau^\flat,\qquad \star\mskip1mu\tau^\flat=-\nu^\flat\ \text{ along }\ \dSig,
\end{equation}
so that $\star$ exchanges the two boundary conditions:
\begin{equation}\label{eq:star-exchanges}
\iota_\nu\mskip1mu\omega=0\quad\Longleftrightarrow\quad\iota_\tau(\star\mskip1mu\omega)=0.
\end{equation}
In particular, if $\omega$ is tangential then $\star\mskip1mu\omega$ is a multiple of $\nu^\flat$ along $\dSig$, so its pullback to each boundary curve vanishes identically and
\begin{equation}\label{eq:bdry-integral-zero}
\int_\gamma\star\mskip1mu\omega=0\ \text{ for every boundary component }\gamma\subseteq\dSig.
\end{equation}

We denote by $\HH^1_T(\Sig)$ the space of harmonic one-forms on $\Sig$ that are smooth up to $\dSig$ and tangential along it, and similarly for $\HH^1_T(\Sgb)$ when $\Sgb$ is compact.

Suppose now that $z=x+i\mskip1mu y$ is a local holomorphic coordinate. For a harmonic one-form $\omega$, the \emph{complexification}
\[
\alpha:=\omega+i\star\omega
\]
is a holomorphic one-form. Indeed, writing $\omega=P\dd x+Q\dd y$, the rotation rule \eqref{eq:star} gives $\star\mskip1mu\omega=-Q\dd x+P\dd y$, hence
\begin{equation}\label{eq:alpha-h}
\alpha=(P-i\mskip1mu Q)(\!\dd x+i\dd y)=h(z)\dd z,\qquad h:=P-i\mskip1mu Q,
\end{equation}
and the two harmonicity equations combine into the Cauchy--Riemann equation
$\p_{\bar z}h=0$. Conversely, the real and imaginary parts of a holomorphic
one-form are harmonic. We shall use freely the resulting dictionary
\begin{equation}\label{eq:complexification}
\omega=\Real(\alpha),\qquad \star\mskip1mu\omega=\Imag(\alpha),\qquad\alpha=h(z)\dd z\text{ holomorphic}.
\end{equation}

\subsection{Conformal invariance of the weighted energy}\label{subsec:conf-inv}

The next lemma shows that the weighted energy of a one-form depends only on the conformal class of the metric, provided the weight is regarded as a function and not as a density.

\begin{lemma}\label{lem:conf-inv}
Let $h$ and $\tilde h=e^{2u}h$ be conformally related metrics on $\Sig$, with\linebreak $u\in C^\infty(\Sig)$, and let $\rho:\Sig\to\R_{>0}$ be a function. Then:
\begin{enumerate}[label=\rm(\roman*),itemsep=3pt,topsep=4pt]
\item $\star_{\tilde h}=\star_h$ on one-forms; consequently, a one-form is harmonic for $h$ if and only if it is harmonic for $\tilde h$;
\item the tangentiality condition $\iota_\nu\omega=0$ is the same for both metrics;
\item the weighted energy is invariant:
\[
\int_\Sig |\omega|_h^2\mskip1mu\rho\dd\mu_h=\int_\Sig |\omega|_{\tilde h}^2\mskip1mu\rho\dd\mu_{\tilde h}.
\]
\end{enumerate}
In particular, the space of weighted tangential harmonic one-forms depends only on the conformal class of $h$ and on $\rho$ as a function.
\end{lemma}

\begin{proof}
(i) If $\{e^1,e^2\}$ is a positively oriented $h$-orthonormal coframe then so is $\{e^ue^1,e^ue^2\}$ for $\tilde h$. Applying the rotation rule \eqref{eq:star} in either coframe to
\[
\omega=\omega_1e^1+\omega_2e^2=(e^{-u}\omega_1)(e^ue^1)+(e^{-u}\omega_2)(e^ue^2)
\]
yields the same form $\omega_1e^2-\omega_2e^1$. Since $d$ is metric independent and $d^\ast\omega=0$ is equivalent to $d(\star\mskip1mu\omega)=0$, both harmonicity equations are unchanged. (ii) The conormals are related by $\tilde\nu=e^{-u}\nu$. (iii) In dimension two, $|\omega|_{\tilde h}^2=e^{-2u}|\omega|_h^2$ and $\dd\mu_{\tilde h}=e^{2u}\dd\mu_h$, and the conformal factors cancel.
\end{proof}

\begin{remark}\label{rem:middle-degree}
Item (iii) holds for one-forms on surfaces and, more generally, for forms of middle degree. In particular, when computing the weighted energy near a puncture, we may replace the given metric by the flat metric $|\!\dd z|^2$ of the coordinate.
\end{remark}

\subsection{The unweighted dimension}\label{subsec:unweighted}

We record the classical count on the compactification, which will reappear in \S\ref{subsec:dim-formula} as the computation of a kernel: 
\begin{equation}\label{eq:dim-unweighted}
\dim_\R\HH^1_T(\Sgb)=2g+k-1.
\end{equation}
By the Hodge--Morrey--Friedrichs decomposition, $\HH^1_T(\Sgb)$ represents the absolute de Rham cohomology $H^1(\Sgb;\R)$ \cite[Ch.~2]{Schwarz}, whose dimension is the first Betti number $2g+k-1$ of a compact orientable surface of genus $g$ with $k$ boundary components; see \cite{ACS2,Sargent} for this count in the present geometric context.

\subsection{Weighted spaces, admissible weights and local orders}\label{subsec:admissible}

\begin{definition}\label{def:weighted-space}
Let $\rho:\Sig\to\R_{>0}$ be a continuous function. The space of \emph{weighted tangential harmonic one-forms} is
\[
\Lrho(\Sig):=\Bigl\{\omega\in\HH^1_T(\Sig):\ \|\omega\|_\rho<+\infty\Bigr\},\qquad\|\omega\|_\rho^2:=\int_\Sig|\omega|^2\rho\dd\mu.
\]
By Lemma~\ref{lem:conf-inv}, this space depends only on the conformal structure of $\Sig$ and on $\rho$ as a function, and not on the auxiliary metric $h$, which we therefore suppress from the notation.
\end{definition}

We next isolate the class of weights for which the dimension can be computed, and
the integer that such a weight attaches to each puncture.

\begin{definition}\label{def:admissible}
The weight $\rho$ is \emph{admissible} if:
\begin{enumerate}[label=\rm(\alph*),itemsep=3pt,topsep=4pt]
\item $\rho$ is bounded on $\Sig$ and bounded away from $0$ on $\Sgb\setminus\bigcup_jD_j$;
\item near each puncture, $\rho$ is comparable to a radial profile: there exist constants $0<c\le C$ and positive measurable functions $\varrho_j$ on $(0,R_1]$ with
\[
c\mskip2mu\varrho_j(|z|)\ \le\ \rho\mskip1mu(z)\ \le\ C\mskip1mu\varrho_j(|z|)\ \text{ on }\ \{0<|z|\le R_1\}\subset D_j^\ast;
\]
\item for every $j$, the \emph{local order}
\begin{equation}\label{eq:Nj-def}
N_j:=\max\Bigl\{m\in\Z_{\ge0} : \textstyle\int_0^{R_1}r^{-2m}\,\varrho_j(r)\,r\dd r<+\infty\Bigr\}
\end{equation}
is finite.
\end{enumerate}
\end{definition}

The integrand in \eqref{eq:Nj-def} is exactly the weighted energy density of the model form $z^{-m}\dd z$ in polar coordinates. Indeed, by \eqref{eq:complexification} and Remark~\ref{rem:middle-degree}, we may compute in the flat coordinate metric, in which
\[
|z^{-m}\dd z|^2\rho\dd\mu=|z|^{-2m}\rho\dd x\dd y\ \asymp\ r^{-2m}\varrho_j(r)\, r\dd r\dd\theta.
\]
Thus $N_j$ is the largest $m$ such that $z^{-m}\dd z$ has finite weighted energy near~$p_j$.

\begin{lemma}\label{lem:Nj-well-def}
The local order is well defined:
\begin{enumerate}[label=\rm(\roman*),itemsep=3pt,topsep=4pt]
\item the set of $m\in\Z_{\ge0}$ for which the integral in \eqref{eq:Nj-def} converges is exactly $\{0,1,\dots,N_j\}$;
\item $N_j$ does not depend on the holomorphic coordinate chosen at $p_j$.
\end{enumerate}
\end{lemma}

\begin{proof}
(i) Shrinking $R_1$ we may assume $R_1<1$; then convergence for $m$ implies convergence for every $m'\le m$ by comparison, so the set of convergent exponents is an initial segment of $\Z_{\ge0}$, and (c) makes it $\{0,1,\dots,N_j\}$. (ii) If $w=\varphi(z)$ is another holomorphic coordinate centered at $p_j$, then $|\varphi(z)|\asymp|z|$ and $|\varphi'(z)|\asymp1$ near $z=0$, so $|w^{-m}\dd w|^2\asymp|z^{-m}\dd z|^2$ and the convergence in \eqref{eq:Nj-def} is unaffected.
\end{proof}

\begin{remark}[Examples]\label{rem:examples}
(1) The trivial weight $\rho\equiv1$ is admissible with $N_j=0$ at every puncture, the integral $\int_0^{R_1}r^{1-2m}\dd r$ converging only for $m=0$. (2) The power profile $\varrho_j(r)=r^{2\kappa}$, $\kappa\in\Z_{\ge0}$, gives $N_j=\kappa$. (3) The borderline profile $\varrho_j(r)=r^{2\kappa}\bigl(\log(1/r)\bigr)^{-2}$ gives $N_j=\kappa+1$ (Remark~\ref{rem:log-borderline}); this is the case of the weight \eqref{eq:rho-def}, see Lemma~\ref{lem:N-computation}.
\end{remark}

\subsection{Riemann--Roch on the Schottky double}\label{subsec:double}

\begin{remark}\label{rem:double}
The formula of Theorem~\ref{thm:dim-formula} can also be obtained from the Riemann--Roch theorem. Let $\Sdouble$ be the Schottky double of $\Sgb$, that is, the closed Riemann surface of genus $2g+k-1$ obtained by gluing $\Sgb$ and a copy of it with the opposite orientation along $\dSgb$, and let $\sigma$ be the corresponding antiholomorphic involution of $\Sdouble$ \cite{SchifferSpencer}. By the reflection principle, harmonic one-forms on $\Sgb$ which are tangential along $\dSgb$ correspond to holomorphic one-forms $\alpha$ on $\Sdouble$ with $\sigma^\ast\alpha=\overline\alpha$, and a pole of order $N_j$ at $p_j$ corresponds to poles of order $N_j$ at $p_j$ and at $\sigma(p_j)$. Applying the Riemann--Roch theorem to the meromorphic one-forms on $\Sdouble$ with poles bounded by the divisor $\sum_jN_j\bigl(p_j+\sigma(p_j)\bigr)$, and taking the real subspace given by the condition $\sigma^\ast\alpha=\overline\alpha$, one obtains exactly $2g+k-1+2\sum_jN_j-\varepsilon$. Moreover, since $\sigma$ is antiholomorphic, the residues of $\alpha$ at $p_j$ and at $\sigma(p_j)$ are complex conjugate, and the residue theorem on $\Sdouble$ gives $\sum_j\Real(c_{-1,j})=0$, which is the relation of Lemma~\ref{lem:residue-necessity}. We do not use the double in the proof below.
\end{remark}

\section{The dimension of the weighted space}\label{sec:weighted}

We keep the setting of \S\ref{subsec:setting} and fix an admissible weight $\rho$ with local orders $N_1,\dots,N_r$. In this section we prove Theorem~\ref{thm:dim-formula}.

\subsection{Local analysis: the Laurent expansion}\label{subsec:laurent}

Fix a puncture $p_j$ and a holomorphic coordinate $z$ as in \S\ref{subsec:setting}. If $\omega\in\Lrho(\Sig)$, then by \S\ref{subsec:harmonic} the complexification $\alpha=\omega+i\star\omega$ is holomorphic on the punctured disk $D_j^\ast$, so it has a Laurent expansion
\begin{equation}\label{eq:laurent-series}
\alpha=h(z)\dd z,\qquad h(z)=\sum_{n=-\infty}^{+\infty}c_nz^n,
\end{equation}
convergent on $0<|z|<R_1$, \emph{a priori} with an essential singularity at $z=0$. The next lemma shows that weighted integrability truncates the principal part.

\begin{lemma}\label{lem:laurent}
Let $\rho$ be admissible and let $\omega\in\Lrho(\Sig)$. Then, at each puncture~$p_j$,
\[
c_n=0\ \text{ for every }\ n<-N_j.
\]
Consequently, $\alpha$ extends to a meromorphic one-form on all of $\Sgb$, with a pole of order at most $N_j$ at $p_j$ and no other singularity.
\end{lemma}

\begin{proof}
By Lemma~\ref{lem:conf-inv}(iii), we may compute the weighted energy on $D_j^\ast$ in the flat metric $|\!\dd z|^2$ of the coordinate. In that metric, writing $h=P-i\mskip1mu Q$ as in \eqref{eq:alpha-h}, one has $\omega=P\dd x+Q\dd y$ and therefore
\[
|\omega|^2\dd\mu=\bigl(P^2+Q^2\bigr)\dd x\dd y=|h(z)|^2\dd x\dd y.
\]
By Definition~\ref{def:admissible}(b), we may moreover replace $\rho$ by the radial profile $\varrho_j$, at the cost of the universal constants $c$ and $C$. Finiteness of $\|\omega\|_\rho$ near $p_j$ is thus equivalent to
\begin{equation}\label{eq:finite-energy-disk}
\int_{\{0<|z|<R_1\}}|h(z)|^2\mskip1mu\varrho_j(|z|)\dd x\dd y<+\infty.
\end{equation}
In polar coordinates, $\dd x\dd y=r\dd r\dd\theta$; on each circle $|z|=r$, the Laurent series \eqref{eq:laurent-series} converges uniformly and is the Fourier series of $\theta\mapsto h(re^{i\theta})$, so Parseval's identity and Tonelli's theorem turn \eqref{eq:finite-energy-disk} into
\[
2\pi\sum_{n\in\Z}|c_n|^2\int_0^{R_1}r^{2n+1}\mskip1mu\varrho_j(r)\dd r<+\infty.
\]
Every summand being nonnegative, for each $n$ either the radial integral converges or $c_n=0$. For $n=-m$ with $m\ge1$, the radial integral is $\int_0^{R_1}r^{-2m}\varrho_j(r)\,r\dd r$, exactly the integral in the definition \eqref{eq:Nj-def} of $N_j$, which converges precisely when $m\le N_j$ by Lemma~\ref{lem:Nj-well-def}(i). Therefore $c_n=0$ for every $n<-N_j$, and $z=0$ is a pole of $h$ of order at most $N_j$. The puncture being arbitrary, $\alpha$ extends meromorphically across every $p_j$ with the stated bounds; it is holomorphic up to $\dSgb$, where $\omega$ is a smooth harmonic form.
\end{proof}

For $\rho\equiv1$ all $N_j$ vanish (Remark~\ref{rem:examples}(1)), and Lemma~\ref{lem:laurent} says that every square integrable harmonic one-form extends smoothly across the punctures, which is a classical fact.

\begin{remark}[The logarithmic borderline]\label{rem:log-borderline}
We verify Remark~\ref{rem:examples}(3), which is the case used in Section~\ref{sec:fbms}. For $\varrho_j(r)=r^{2\kappa}(\log(1/r))^{-2}$, $\kappa\in\Z_{\ge0}$, the integral in \eqref{eq:Nj-def} is $\int_0^{R_1}r^{2\kappa-2m+1}(\log(1/r))^{-2}\dd r$, with $R_1<1$. For $m\le\kappa$ the exponent of $r$ is greater than $-1$ and the integral converges. For $m=\kappa+1$ the exponent is $-1$ and, with $t=\log(1/r)$, the integral becomes $\int^{+\infty}t^{-2}\dd t<+\infty$. For $m\ge\kappa+2$ the integrand is bounded below by $r^{-3}(\log(1/r))^{-2}$ near $r=0$, and the integral diverges. Hence $N_j=\kappa+1$.
\end{remark}

\subsection{Counting the local degrees of freedom}\label{subsec:local-count}

By Lemma~\ref{lem:laurent}, an element of $\Lrho(\Sig)$ may carry at $p_j$ a singular part
\begin{equation}\label{eq:hsing}
h_{\mathrm{sing}}(z)=\frac{c_{-1}}{z}+\frac{c_{-2}}{z^2}+\cdots+\frac{c_{-N_j}}{z^{N_j}}.
\end{equation}
We now count how many real parameters this represents. If $N_j=0$, the singular part is empty and the count is $0$, so we assume $N_j\ge1$.

\begin{lemma}\label{lem:local-count}
The space of possible singular parts \eqref{eq:hsing} at $p_j$ has real dimension $2N_j$. More precisely, the real one-forms $\Real(z^{-\ell}\dd z)$ and $\Imag(z^{-\ell}\dd z)$, for $1\le\ell\le N_j$, are linearly independent over $\R$.
\end{lemma}

\begin{proof}
Each coefficient decomposes as $c_{-\ell}=a_{-\ell}+i\mskip1mu b_{-\ell}$ with $a_{-\ell},b_{-\ell}\in\R$, giving \emph{a priori} $2N_j$ real parameters, and the assignment taking $h_{\mathrm{sing}}$ to the real one-form $\Real\bigl(h_{\mathrm{sing}}(z)\dd z\bigr)$ is $\R$-linear; the content of the lemma is its injectivity. In the notation \eqref{eq:alpha-h}, the real part of $h\dd z$ is the one-form $P\dd x+Q\dd y$ with $h=P-i\mskip1mu Q$, so it vanishes identically if and only if $h$ does; and a Laurent polynomial vanishes if and only if all of its coefficients do. Therefore, since
\begin{equation}\label{eq:coefficients_ab}
a_{-\ell}\Real(z^{-\ell}\dd z)-b_{-\ell}\Imag(z^{-\ell}\dd z)=\Real\bigl((a_{-\ell}+i\mskip1mu b_{-\ell})z^{-\ell}\dd z\bigr),
\end{equation}
because $\Imag(\alpha)=\Real(-i\mskip1mu\alpha)$, the one-forms $\Real(z^{-\ell}\dd z)$ and $\Imag(z^{-\ell}\dd z)$ are linearly independent over $\R$ for $1\le\ell\le N_j$, and the space of singular parts has real dimension $2N_j$.
\end{proof}

For later use, we record the resulting model forms in polar coordinates. From $\dd z=e^{i\theta}(\!\dd r+i\mskip1mu r\dd\theta)$, one computes
\begin{equation}\label{eq:model-forms}
\begin{aligned}
\omega_{-\ell}^{(a)}&:=\Real\bigl(z^{-\ell}\dd z\bigr)=\frac{1}{r^{\ell}}\cos((1-\ell)\theta)\dd r-\frac{1}{r^{\ell-1}}\sin((1-\ell)\theta)\dd\theta,\\[4pt]
\omega_{-\ell}^{(b)}&:=-\Imag\bigl(z^{-\ell}\dd z\bigr)=-\frac{1}{r^{\ell}}\sin((1-\ell)\theta)\dd r-\frac{1}{r^{\ell-1}}\cos((1-\ell)\theta)\dd\theta,
\end{aligned}
\end{equation}
so that the singular part of $\omega=\Real(\alpha)$ at $p_j$ reads $\sum_{\ell=1}^{N_j}\bigl(a_{-\ell}\mskip2mu\omega_{-\ell}^{(a)}+b_{-\ell}\mskip2mu\omega_{-\ell}^{(b)}\bigr)$, by identity \eqref{eq:coefficients_ab}.

\subsection{The global residue constraint is necessary}\label{subsec:necessity}

So far, we have fixed one puncture at a time. From now on, the statements are global, relating the data at all punctures simultaneously, so we write $c_{-\ell,j}$, and correspondingly $a_{-\ell,j}$ and $b_{-\ell,j}$, for the coefficients of the expansion at $p_j$. Recall that $c_{-1,j}$ is the \emph{residue} of $\alpha$ at $p_j$. Note that the variable $a_{-1,j}$ exists only when $N_j\ge1$; a puncture with $N_j=0$ contributes no residue variable at all.

\begin{lemma}\label{lem:residue-necessity}
Let $\rho$ be admissible and let $\omega\in\Lrho(\Sig)$, with complexification $\alpha=\omega+i\star\omega$. Then the real parts of the residues satisfy
\begin{equation}\label{eq:residue-relation}
\sum_{j\,:\,N_j\ge1}a_{-1,j}=0,\ \text{ where }\ a_{-1,j}:=\Real(c_{-1,j}).
\end{equation}
\end{lemma}

\begin{proof}
Choose a radius $0<r_0<R_1$ and let
\[
\Sig_{r_0}:=\Sgb\setminus\bigcup_{j=1}^rD_j(r_0),\qquad D_j(r_0):=\{|z|<r_0\}\subset D_j,
\]
be the compact surface with boundary obtained by removing the open coordinate disks of radius $r_0$ around the punctures, so that
\[
\p\Sig_{r_0}=\dSgb\cup C_1\cup\cdots\cup C_r\ \text{ with }\ C_j:=\p D_j(r_0).
\]
We denote by $C_j^+$ the circle $C_j$ with its counterclockwise parametrization $z=r_0e^{i\theta}$, which is opposite to the orientation induced from $\Sig_{r_0}$, and write Stokes' theorem accordingly.

The form $\star\mskip1mu\omega$ is smooth on $\Sig_{r_0}$ and closed there by \eqref{eq:star-omega-closed}, so Stokes' theorem gives
\begin{equation}\label{eq:stokes}
0=\int_{\Sig_{r_0}}d(\star\mskip1mu\omega)=\int_{\dSgb}\star\mskip1mu\omega-\sum_{j=1}^r\int_{C_j^+}\star\mskip1mu\omega,
\end{equation}
the minus sign accounting for the induced orientation of the inner circles. The outer integral vanishes: $\omega$ is tangential, so $\star\mskip1mu\omega$ is normal along $\dSgb$ by \eqref{eq:star-exchanges}, its pullback to each boundary curve is the zero one-form, and \eqref{eq:bdry-integral-zero} applies. Hence
\begin{equation}\label{eq:stokes-sum}
\sum_{j=1}^r\int_{C_j^+}\star\mskip1mu\omega=0.
\end{equation}

It remains to evaluate each circle integral. The Laurent series \eqref{eq:laurent-series} converges uniformly on $C_j^+$, so term-by-term integration and the residue theorem give
\[
\int_{C_j^+}\alpha=2\pi i\mskip1mu c_{-1,j}.
\]
Since $\star\omega=\Imag(\alpha)$ by \eqref{eq:complexification}, and since taking imaginary parts commutes with pulling back and integrating,
\begin{equation}\label{eq:period}
\int_{C_j^+}\star\mskip1mu\omega=\Imag\bigg(\int_{C_j^+}\alpha\bigg)=\Imag\bigl(2\pi i\mskip1mu c_{-1,j}\bigr)=2\pi\Real(c_{-1,j})=2\pi\mskip1mu a_{-1,j}.
\end{equation}
If $N_j=0$, there is no residue term and the integral vanishes, in accordance with the convention that the variable $a_{-1,j}$ is then absent. Substituting \eqref{eq:period} into \eqref{eq:stokes-sum} gives \eqref{eq:residue-relation}.
\end{proof}

\begin{remark}[Asymmetry of the constraint]\label{rem:asymmetry}
The relation \eqref{eq:residue-relation} involves only the real parts of the residues. This comes from the boundary condition: we apply Stokes' theorem to $\star\mskip1mu\omega$, which is closed by \eqref{eq:star-omega-closed}, and the tangential condition makes $\int_{\dSgb}\star\mskip1mu\omega$ vanish. If $\omega$ were normal along $\dSgb$, we would apply Stokes' theorem to $\omega$ instead, and since
\[
\int_{C_j^+}\omega=\Real\bigl(2\pi i\mskip1mu c_{-1,j}\bigr)=-2\pi b_{-1,j},
\]
we would obtain $\sum_j b_{-1,j}=0$. See also Remark~\ref{rem:double}.
\end{remark}

\subsection{The global residue constraint is sufficient}\label{subsec:sufficiency}

We now prove that relation \eqref{eq:residue-relation} is the only obstruction to prescribing the principal parts. By Lemma~\ref{lem:residue-necessity}, the hypothesis $\sum_ja_{-1,j}=0$ below cannot be removed.

\begin{lemma}\label{lem:sufficiency}
Let $\rho$ be admissible and let real numbers 
\[
a_{-\ell,j},\, b_{-\ell,j},\ \text{ with }\ 1\le\ell\le N_j \ \text{ and }\ 1\le j\le r,
\]
be given, subject to the residue relation 
\[
\sum_{j:N_j\ge1}a_{-1,j}=0.
\]
Then there exists $\omega\in\Lrho(\Sig)$ whose complexification $\alpha=\omega+i\star\omega$ has Laurent principal part 
\[
\sum_{\ell=1}^{N_j}\bigl(a_{-\ell,j}+i\mskip1mu b_{-\ell,j}\bigr)z^{-\ell}
\]
at every puncture $p_j$, and no other singularity.
\end{lemma}

\begin{proof}
The proof is divided into four steps. Step~1 builds a reference form $\omega_0$ by truncating the prescribed local models; Step~2 isolates the harmonicity defects created by the truncation; Step~3 removes them by a smooth correction $\eta$ built from two scalar potentials, the residue hypothesis entering as the compatibility condition of a Neumann problem; Step~4 assembles $\omega=\omega_0+\eta$. Throughout, $\mathring\Sig:=\Sgb\setminus\dSgb$ is the interior of $\Sgb$ as a manifold with boundary; the punctures belong to it, and the interior equations below are required to hold across them.

\medskip
\noindent\textbf{Step 1: the singular reference form.}
For each $j$ with $N_j\ge1$ and each $1\le\ell\le N_j$, let $\omega_{-\ell,j}^{(a)}$ and $\omega_{-\ell,j}^{(b)}$ denote the model forms \eqref{eq:model-forms}, expressed in polar coordinates centered at $p_j$. They are smooth, closed and coclosed on $D_j^\ast$ and are, up to sign, the real and imaginary parts of the holomorphic form $z^{-\ell}\dd z$. Fix $\epsilon_0>0$ smaller than half the radius of every $D_j$ and than half the distance between any two punctures and between any puncture and $\dSgb$. Let $\chi\in C^\infty([0,\infty))$ satisfy $\chi\equiv1$ on $[0,\epsilon_0/2]$, $\chi\equiv0$ on $[\epsilon_0,\infty)$, and $0\le\chi\le1$, and define, on $\mathring\Sig\setminus\{p_1,\dots,p_r\}$, 
\[
\omega_0:=\sum_{j=1}^r\chi(r_j)\sum_{\ell=1}^{N_j}\Bigl(a_{-\ell,j}\,\omega_{-\ell,j}^{(a)}+b_{-\ell,j}\,\omega_{-\ell,j}^{(b)}\Bigr),
\]
where $r_j$ is the radial coordinate centered at $p_j$, the inner sum is empty when $N_j=0$, and $\omega_0$ is extended by $0$ outside $\bigcup_jD_j(\epsilon_0)$. By construction, $\omega_0$ agrees with the
prescribed model on $\{0<r_j<\epsilon_0/2\}$ and vanishes on $\{r_j>\epsilon_0\}$ for every~$j$; in particular, $\omega_0\equiv0$ on a neighborhood of $\dSgb$.

\medskip
\noindent\textbf{Step 2: the harmonicity defects.}
Define $F_2:=d\mskip1mu\omega_0$ and $F_0:=d^\ast\omega_0$. Both vanish on $\{0<r_j<\epsilon_0/2\}$, where $\omega_0$ coincides with the closed and coclosed model, and on $\{r_j>\epsilon_0\}$, where $\omega_0\equiv0$; hence they are smooth, supported in the annuli $A_j:=\{\epsilon_0/2\le r_j\le\epsilon_0\}$, and their extensions by zero define a smooth two-form $F_2$ and a smooth function $F_0$ on all of $\Sgb$.

\medskip
\noindent\textbf{Step 3: removing the defects by scalar potentials.}
We seek a one-form $\eta$, smooth on all of $\Sgb$, punctures included, solving
\begin{equation}\label{eq:target-system}
\begin{cases}
d\mskip1mu\eta=-F_2 & \text{on }\mathring\Sig,\\
d^\ast\eta=-F_0 & \text{on }\mathring\Sig,\\
\iota_\nu\eta=0 & \text{on }\dSgb.
\end{cases}
\end{equation}
Write $F_2=\widehat F_2\dd\mu$ with $\widehat F_2\in C^\infty(\Sgb)$ and make the ansatz
\[
\eta=d f+\star\mskip1mu dg,\qquad f,g\in C^\infty(\Sgb).
\]
By the elementary two-dimensional identities,
\begin{equation}\label{eq:2d-identities}
d(df)=0, \qquad d^\ast(\star dg)=0, \qquad d(\star dg)=(\Delta g)\dd\mu, \qquad d^\ast(df)=-\Delta f,
\end{equation}
with $\Delta=\diver\nabla$, which follow from $d^2=0$, $d^\ast=-\star\! d\star$, $\star\star=-\mathrm{id}$ and $d^\ast d=-\Delta$ on one-forms, one has $d\eta=(\Delta g)\dd\mu$ and $d^\ast\eta=-\Delta f$, so the two interior equations in \eqref{eq:target-system} hold if and only if
\[
\Delta g=-\widehat F_2\ \text{ and }\ \Delta f=F_0 \ \text{ on }\ \mathring\Sig.
\]

For the boundary condition, decomposing $dg$ in the frame $\{\nu,\tau\}$ and applying \eqref{eq:star-frame} gives $\iota_\nu\eta=\p_\nu f-\p_\tau g$ along $\dSgb$, which is satisfied by imposing
\begin{equation}\label{eq:BC}
g=0\ \text{ and } \ \p_\nu f=0\ \text{ on }\ \dSgb,
\end{equation}
the first of which forces $\p_\tau g=0$. Condition \eqref{eq:BC} is stronger than $\p_\nu f=\p_\tau g$, but it gives a Dirichlet problem for $g$ and a Neumann problem for $f$, which can be solved separately.

The Dirichlet problem $\Delta g=-\widehat F_2$ on $\mathring\Sig$, $g=0$ on $\dSgb$, has a unique solution, smooth on $\Sgb$ by elliptic regularity, since by the maximum principle $0$ is not a Dirichlet eigenvalue; no compatibility condition is required. The Neumann problem $\Delta f=F_0$ on $\mathring\Sig$, $\p_\nu f=0$ on $\dSgb$, is solvable, by the Fredholm alternative for the Neumann Laplacian, whose kernel consists of constant functions, if and only if
\begin{equation}\label{eq:compat}
\int_{\Sgb}F_0\dd\mu=0,
\end{equation}
and the solution is then smooth and unique up to an additive constant, which is irrelevant since only $df$ enters $\eta$. We verify \eqref{eq:compat} by the residue computation of \S\ref{subsec:necessity}. Let $\Sgb_{r_0}=\Sgb\setminus\bigcup_jD_j(r_0)$ with $r_0<\epsilon_0/2$. Since $F_0$ vanishes on the removed disks and $F_0\dd\mu=-d(\star\mskip1mu\omega_0)$ away from the punctures by \eqref{eq:2d-identities}, Stokes' theorem gives, exactly as in \eqref{eq:stokes},
\[
\int_{\Sgb}F_0\dd\mu=-\int_{\dSgb}\star\mskip1mu\omega_0+\sum_{j=1}^r\int_{C_j^+}\star\mskip1mu\omega_0.
\]
The outer integral vanishes because $\omega_0\equiv0$ near $\dSgb$; on $C_j^+$ the form $\omega_0$ coincides with the prescribed model, so \eqref{eq:period} applies verbatim and gives $\int_{C_j^+}\star\mskip1mu\omega_0=2\pi a_{-1,j}$. Hence 
\[
\int_{\Sgb}F_0\dd\mu=2\pi\sum_{j:N_j\ge1}a_{-1,j}=0
\]
by the residue hypothesis, and the Neumann problem is solvable.

\medskip
\noindent\textbf{Step 4: synthesis.} 
Set $\eta:=df+\star\,dg$ and $\omega:=\omega_0+\eta$ on $\mathring\Sig\setminus\{p_1,\dots,p_r\}$. By Step~3, $d\mskip1mu\omega=F_2-F_2=0$ and $d^\ast\omega=F_0-F_0=0$, up to $\dSgb$ by continuity, so $\omega$ is harmonic; near $\dSgb$ one has $\omega=\eta$, which is tangential by \eqref{eq:BC}. Near each $p_j$, the form $\eta$ is smooth, as $f,g\in C^\infty(\Sgb)$, while $\omega_0$ equals the prescribed model exactly, so that the Laurent principal part of $\alpha=\omega+i\star\omega$ at $p_j$ is $\sum_{\ell=1}^{N_j}(a_{-\ell,j}+i\mskip1mu b_{-\ell,j})z^{-\ell}$ and $\omega$ has no singularity at the punctures with $N_j=0$. Finally, away from the punctures, $\omega$ and $\rho$ are bounded, while near $p_j$, the form $\omega$ differs from the model by the bounded $\eta$ and each model term with $\ell\le N_j$ has finite $\rho$-energy by the radial computation of Lemma~\ref{lem:laurent}; hence $\|\omega\|_\rho<+\infty$ and $\omega\in\Lrho(\Sig)$ realizes the prescribed data. 
\end{proof}

\subsection{The dimension formula}\label{subsec:dim-formula}

We now assemble the four lemmas. Consider the \emph{principal part map}
\begin{align*}
P:\Lrho(\Sig)&\longrightarrow\bigoplus_{j=1}^r\R^{2N_j},\\
\omega&\longmapsto\bigl(a_{-\ell,j}(\omega),b_{-\ell,j}(\omega)\bigr)_{1\le\ell\le N_j,\, 1\le j\le r},
\end{align*}
which records the real and imaginary parts of the singular Laurent coefficients of $\alpha=\omega+i\star\omega$ at every puncture; it is a linear map that is well defined by Lemma~\ref{lem:laurent}.

\begin{lemma}\label{lem:kernel-P}
The kernel of $P$ is
\[
\ker P=\bigl\{\omega\in\Lrho(\Sig):\alpha\text{ extends holomorphically across every }p_j\bigr\},
\]
and this space is isomorphic to $\HH^1_T(\Sgb)$. Consequently,
\[
\dim_\R\ker P=2g+k-1.
\]
\end{lemma}

\begin{proof}
By Lemma~\ref{lem:laurent}, $\omega\in\ker P$ means $c_{n,j}=0$ for all $n<0$ and all $j$, so $\alpha$, and with it $\omega=\Real\alpha$, extends smoothly across every puncture to a tangential harmonic form on $\Sgb$. Conversely, the restriction to $\Sig$ of any element of $\HH^1_T(\Sgb)$ is harmonic, tangential and bounded, hence $\rho$-integrable, $\rho$ being bounded and $\Sgb$ of finite area, with vanishing principal parts. Restriction and extension are mutually inverse, so $\ker P\cong\HH^1_T(\Sgb)$ and \eqref{eq:dim-unweighted} applies.
\end{proof}

We are now in position to prove the dimension formula announced in the introduction, which we restate for convenience.

\begin{thmdim}
Let $\Sig=\Sgb\setminus\{p_1,\dots,p_r\}$ be as in \S\ref{subsec:setting} and let $\rho$ be an admissible weight with local orders $N_1,\dots,N_r$. Then
\begin{equation}\label{eq:dim-formula}
\dim_\R\Lrho(\Sig)=2g+k-1+2\sum_{j=1}^rN_j-\varepsilon,
\end{equation}
where
\[
\varepsilon=
\begin{cases}
1,&\text{if }\ \sum_{j}N_j\ge1,\\[4pt]
0,&\text{if }\ N_1=\cdots=N_r=0.
\end{cases}
\]
\end{thmdim}

\begin{proof}[Proof of Theorem~\ref{thm:dim-formula}]
By the rank--nullity theorem,
\[
\dim_\R\Lrho(\Sig)=\dim_\R\ker P+\dim_\R\im P,
\]
the kernel being computed in Lemma~\ref{lem:kernel-P}. It remains to compute the image exactly. Define the \emph{residue functional}
\[
\lambda:\bigoplus_{j=1}^r\R^{2N_j}\to\R,\qquad\lambda\bigl((a_{-\ell,j},b_{-\ell,j})_{\ell,j}\bigr)=\sum_{j\,:\,N_j\ge1}a_{-1,j},
\]
which reads off the sum of the real residue coordinates. Observe that by Lemma~\ref{lem:residue-necessity}, $\im P\subseteq\ker\lambda$; and by Lemma~\ref{lem:sufficiency}, $\im P\supseteq\ker\lambda$. Hence $\im P=\ker\lambda$ and
\[
\dim_\R\im P=2\sum_jN_j-\rank\lambda.
\]
If $\sum_jN_j\ge1$, then $\lambda$ is onto $\R$ and $\rank\lambda=1$; if all $N_j=0$, its domain is the zero space and $\rank\lambda=0$. In both cases, $\rank\lambda=\varepsilon$, and adding the two dimensions gives \eqref{eq:dim-formula}.
\end{proof}

\begin{remark}[The case $\varepsilon=0$]\label{rem:epsilon}
When $N_1=\cdots=N_r=0$, as for $\rho\equiv1$, there are no residues and no residue relation, and the dimension is $2g+k-1$; so the formula $2g+k+2\sum_jN_j-2$ is not correct in this case. In Section~\ref{sec:fbms} this case does not occur, since $N_j=d_j+1\ge2$ for every end.
\end{remark}

\section{The weighted space of a free boundary minimal surface}\label{sec:fbms}

Throughout this section, $\Sig$ is a complete, noncompact, orientable, immersed free boundary minimal surface of finite Morse index in $\Om\subset\R^3$, with compact boundary $\dSig\subset\dOm$, of genus $g$, with $k$ boundary components and $r$ ends of multiplicities $d_1,\dots,d_r$. By \eqref{eq:model}, it is of the form \eqref{eq:setting}, so Section~\ref{sec:weighted} applies verbatim; the weight is $\rho$ of \eqref{eq:rho-def}, and we write $\Lstar(\Sig)=\Lrho(\Sig)$ for this choice.

\subsection{The local orders of the logarithmic weight}

\begin{lemma}\label{lem:N-computation}
The weight \eqref{eq:rho-def} is admissible on $\Sig$ in the sense of Definition~\ref{def:admissible}, with local orders
\[
N_j=d_j+1\ \text{ at every puncture }\ p_j.
\]
\end{lemma}

\begin{proof}
We verify the three conditions of Definition~\ref{def:admissible}.

\textrm{(a)} The function $\rho$ is bounded above by $\rho(0)=1/(\log2)^2$ on all of $\R^3$. Away from the punctures, that is, on $\Sgb\setminus\bigcup_jD_j$, the corresponding subset of $\Sig$ is compact, the immersion being proper, so $|X|$ is bounded there and $\rho$ is bounded away from~$0$.

\textrm{(b)} Near $p_j$, we insert the asymptotics \eqref{eq:end-consequences} into \eqref{eq:rho-def}, obtaining
\[
\rho\ \asymp\ \frac{1}{r^{-2d_j}\bigl(d_j\log(1/r)\bigr)^2}\ \asymp\ r^{2d_j}\bigl(\log(1/r)\bigr)^{-2}=:\varrho_j(r),
\]
which is a radial profile in the required sense. This is \eqref{eq:rho-profile}, already computed in Lemma~\ref{lem:rho-integrable}.

\textrm{(c)} The profile $\varrho_j$ is of the borderline logarithmic type of Remark~\ref{rem:log-borderline}, with $\kappa=d_j$. By that computation
\[
N_j=\kappa+1=d_j+1<+\infty,
\]
so the local orders are finite and $\rho$ is admissible.
\end{proof}

Note that the weight $(1+|x|^2)^{-1}$ alone would give $N_j=d_j$, by Remark~\ref{rem:examples}(2); the logarithmic factor gives one more order, as in \cite[Prop.~3.1]{CM2}.

\subsection{The dimension}

\begin{corollary}\label{cor:fbms}
With $\Sig$ as above,
\begin{equation}\label{eq:dim-fbms}
\dim_\R\Lstar(\Sig)
=2g+k+2\sum_{j=1}^r(d_j+1)-2.
\end{equation}
In particular, $\dim_\R\Lstar(\Sig)\ge2g+k+2\ge3$.
\end{corollary}

\begin{proof}
By Lemma~\ref{lem:N-computation}, the weight is admissible with $N_j=d_j+1\ge2$. Since $\Sig$ is noncompact, it has at least one end, so $\sum_jN_j\ge2>1$ and therefore $\varepsilon=1$ in Theorem~\ref{thm:dim-formula}. Substituting $N_j=d_j+1$ into \eqref{eq:dim-formula} gives \eqref{eq:dim-fbms}. For the last assertion, use $r\ge1$, $d_1\ge1$, $k\ge1$ and $g\ge0$.
\end{proof}

\begin{remark}[Comparison with the boundaryless case]\label{rem:compare-CM}
For a complete minimal surface without boundary, of finite index, genus $g$ and $r$ ends of multiplicities $d_j$, Chodosh and Máximo \cite[Prop.~3.2]{CM2} compute the corresponding weighted dimension to be $2g+2\sum_j(d_j+1)-2$, via Riemann--Roch on the closed compactification. Formula \eqref{eq:dim-fbms} is the analogue for surfaces with boundary, where $2g$ is replaced by $2g+k-1$; see Remark~\ref{rem:double}.
\end{remark}

\subsection{Optimality of the weight}

\begin{remark}\label{rem:optimality}
A weight decaying faster than $\rho$ does not improve the estimate. Indeed, the weight is used in two places: in the computation of the $N_j$ and in the estimate \eqref{eq:cutoff-vs-rho} for the cut-off. Iterated logarithmic factors do not change the $N_j$: the argument of Remark~\ref{rem:log-borderline} shows that for $m\ge\kappa+2$ the integrand is bounded below by $r^{-3}$ times powers of $\log(1/r)$, and the integral diverges. On the other hand, for a weight $\rho\asymp(1+|x|^2)^{-(1+\delta)}$, $\delta>0$, we have $N_j>d_j+1$ as soon as $\delta>1/d_j$, but there is no cut-off function $\varphi$ with $\varphi\equiv1$ near $\dSig$, compact support and $|\nabla\varphi|^2\le C\rho$. In fact, such a function would require $\int^{+\infty}\sqrt{\rho(s)}\dd s=+\infty$ along each end, which holds for the weight \eqref{eq:rho-def}, since $\sqrt\rho\asymp(s\log s)^{-1}$, but not for any faster power decay.
\end{remark}

\section{The boundary term and the rotational fields}\label{sec:boundary}

Throughout this section, $\Sig$ is as in Section~\ref{sec:fbms}, and we use the $\R^3$-valued second variation \eqref{eq:Q-vec}. For a tangent field $\xi$ along $\dSig$, we write
\begin{equation}\label{eq:B-vector}
\cB\mskip1mu\xi:=D_\nu\xi-\II_{\dOm}(N,N)\xi,
\end{equation}
so that $\cB\mskip1mu\xi=\sum_i(\cB u_i)E_i$ componentwise, in the notation \eqref{eq:coordinates}.

\subsection{The key identity}\label{subsec:key-identity}

The following lemma computes the second variation on the coordinates of a harmonic field tangent to $\dSig$. Note that no integrability condition on $\xi$ is required.

\begin{lemma}\label{lem:key-identity}
Let $\xi\in\HH^1_T(\Sig)$ and let $\varphi\in C_c^\infty(\Sig)$ satisfy $\varphi\equiv1$ on a neighborhood of $\dSig$. Then
\begin{equation}\label{eq:key-identity}
\cQ(\varphi\xi,\varphi\xi)=\int_\Sig|\nabla\varphi|^2|\xi|^2\dd\mu-\int_{\dSig}H_{\dOm}|\xi|^2\dd\sigma,
\end{equation}
where $\varphi\xi=(\varphi u_1,\varphi u_2,\varphi u_3)$ and, by \eqref{eq:Q-vec}, $\cQ(\varphi\xi,\varphi\xi)=\sum_{i=1}^3\cQ(\varphi u_i,\varphi u_i)$.
\end{lemma}

\begin{proof}
Integrating \eqref{eq:Q} by parts, for each $i$,
\[
\begin{aligned}
\cQ(\varphi u_i,\varphi u_i)\,={}&-\int_\Sig\varphi u_i\bigl(\Delta(\varphi u_i)+|A|^2\varphi u_i\bigr)\dd\mu\\[4pt]
&+\int_{\dSig}\bigl(\p_\nu(\varphi u_i)-\II_{\dOm}(N,N)\varphi u_i\bigr)\varphi u_i\dd\sigma.
\end{aligned}
\]
Since $\varphi\equiv1$ near $\dSig$, along $\dSig$ one has $\p_\nu(\varphi u_i)=\p_\nu u_i$ and $\varphi u_i=u_i$, so the boundary integrand is $\bigl(\p_\nu u_i-\II_{\dOm}(N,N)u_i\bigr)u_i$. For the interior integrand, we expand
\[
\Delta(\varphi u_i)+|A|^2\varphi u_i=\varphi\mskip1mu\cJ u_i+u_i\Delta\varphi+2\langle\nabla\varphi,\nabla u_i\rangle
\]
and use $\cJ u_i=-2g_i\langle\nabla\xi,A\rangle$ from \eqref{eq:Jui}. Summing over $i$, the term carrying $\langle\nabla\xi,A\rangle$ becomes a multiple of $\sum_ig_iu_i=\langle\xi,N\rangle=0$ and drops out. With $\sum_iu_i^2=|\xi|^2$ and $\sum_iu_i\nabla u_i=\tfrac12\nabla|\xi|^2$, summation over $i$ gives
\begin{equation}\label{eq:key-1}
\begin{aligned}
\cQ(\varphi\xi,\varphi\xi)\,={}&-\int_\Sig\Bigl(\varphi\mskip1mu\Delta\varphi\mskip1mu|\xi|^2+\varphi\mskip1mu\bigl\langle\nabla\varphi,\nabla|\xi|^2\bigr\rangle\Bigr)\dd\mu\\
&+\int_{\dSig}\sum_{i=1}^3\bigl(\p_\nu u_i-\II_{\dOm}(N,N)u_i\bigr)u_i\dd\sigma.
\end{aligned}
\end{equation}

We treat the two integrals separately. For the interior one, since $\nabla\varphi=0$ near $\dSig$, the divergence theorem applied to the field $\varphi|\xi|^2\nabla\varphi$, after expanding by the Leibniz rule, yields
\begin{equation}\label{eq:key-2}
-\int_\Sig\Bigl(\varphi\mskip1mu\Delta\varphi\mskip1mu|\xi|^2+\varphi\mskip1mu\bigl\langle\nabla\varphi,\nabla|\xi|^2\bigr\rangle\Bigr)\dd\mu=\int_\Sig|\nabla\varphi|^2|\xi|^2\dd\mu.
\end{equation}

For the boundary integrand, note first that, along $\dSig$, the field $\xi$ is tangent to $\dSig$, so $\xi=f\tau$ for a unit tangent $\tau$ and a function $f$. Using $d\mskip1mu\omega(Y,Z)=\langle\nabla_Y\xi,Z\rangle-\langle\nabla_Z\xi,Y\rangle$ together with $d\mskip1mu\omega=0$, applied to $Y=\nu$ and $Z=\xi$,
\begin{equation}\label{eq:key-3}
\sum_{i=1}^3u_i\p_\nu u_i=\tfrac12\p_\nu|\xi|^2=\langle\nabla_\nu\xi,\xi\rangle=\langle\nabla_\xi\xi,\nu\rangle=-\II_{\dOm}(\xi,\xi),
\end{equation}
the last equality because $\dSig\subset\dOm$, $\xi$ is tangent to $\dOm$ and $\nu=\eta$ is the outward unit normal of $\dOm$ along $\dSig$. Since $\{\tau,N\}$ is an orthonormal basis of $T(\dOm)$ along $\dSig$ and $H_{\dOm}=\tr\II_{\dOm}$,
\begin{equation}\label{eq:key-4}
\II_{\dOm}(\xi,\xi)+\II_{\dOm}(N,N)|\xi|^2=H_{\dOm}|\xi|^2,
\end{equation}
because $\II_{\dOm}(\xi,\xi)=f^2\II_{\dOm}(\tau,\tau)=|\xi|^2\II_{\dOm}(\tau,\tau)$. Combining \eqref{eq:key-3} and \eqref{eq:key-4},
\begin{equation}\label{eq:key-5}
\sum_{i=1}^3\bigl(\p_\nu u_i-\II_{\dOm}(N,N)u_i\bigr)u_i=-\II_{\dOm}(\xi,\xi)-\II_{\dOm}(N,N)|\xi|^2=-H_{\dOm}|\xi|^2.
\end{equation}
Substituting \eqref{eq:key-2} and \eqref{eq:key-5} into \eqref{eq:key-1} gives \eqref{eq:key-identity}.
\end{proof}

\subsection{Strict positivity of the boundary term}\label{subsec:strict}

\begin{corollary}\label{cor:strict-positivity}
Suppose $H_{\dOm}\ge0$ along $\dSig$ and $H_{\dOm}(p)>0$ for some $p\in\dSig$. Then
\begin{equation}\label{eq:strict-pos}
\int_{\dSig}H_{\dOm}|\xi|^2\dd\sigma>0\ \text{ for every } \xi\in\Lstar(\Sig)\setminus\{0\}.
\end{equation}
Consequently, there is $R_2\ge R_0$ such that, for the logarithmic cut-off $\lcut_R$ of \S\ref{subsec:cutoff},
\begin{equation}\label{eq:negative-Q}
\cQ(\lcut_R\xi,\lcut_R\xi)<0\ \text{ for all } R\ge R_2 \text{ and all } \xi\in\Lstar(\Sig)\setminus\{0\}.
\end{equation}
\end{corollary}

\begin{proof}
Since $H_{\dOm}>0$ at $p$, there is a nonempty open arc $U\subset\dSig$ on which $H_{\dOm}>0$. If $\xi\in\Lstar(\Sig)\setminus\{0\}$ had $|\xi|\equiv0$ on $U$, then $\omega=\xi^\flat$ would vanish on $U$ and Lemma~\ref{lem:unique-cont} would force $\xi\equiv0$; hence $|\xi|\not\equiv0$ on $U$ and \eqref{eq:strict-pos} holds. For \eqref{eq:negative-Q}, the bilinear form $B(\xi,\zeta):=\int_{\dSig}H_{\dOm}\langle\xi,\zeta\rangle\dd\sigma$ is positive definite on $\Lstar(\Sig)$ by \eqref{eq:strict-pos}, so by compactness of the unit sphere $S=\{\|\xi\|_\rho=1\}$, as $\Lstar(\Sig)$ is finite-dimensional, there is $b_0>0$ with $B(\xi,\xi)\ge b_0$ on $S$. Since $\lcut_R\equiv1$ near $\dSig$, Lemma~\ref{lem:key-identity} and the uniform bound \eqref{eq:cutoff-to-zero-2} give
\[
\cQ(\lcut_R\xi,\lcut_R\xi)=\int_\Sig|\nabla\lcut_R|^2|\xi|^2\dd\mu-B(\xi,\xi)\ \le\ C_0\epsilon_R-b_0\ \text{ for }\ \xi\in S,
\]
with $\epsilon_R\searrow0$ independent of $\xi$. Choosing $R_2$ with $C_0\epsilon_{R_2}<b_0$ gives \eqref{eq:negative-Q} on $S$, hence on $\Lstar(\Sig)\setminus\{0\}$ by homogeneity.
\end{proof}

\subsection{The dimension of \texorpdfstring{$\Lst\cap\Lsq$}{}}\label{subsec:dim-Lstar}

For $\vec a\in\R^3$, let $\ell_{\vec a}:=\langle\vec a,X\rangle$; then $\nabla\ell_{\vec a}=\vec a^\top$ and
\begin{equation}\label{eq:pointwise-Lstar}
\left|\star\nabla\ell_{\vec a}\right|^2=\left|\nabla\ell_{\vec a}\right|^2=|\vec a|^2-\langle\vec a,N\rangle^2.
\end{equation}
Every element of $\Lst(\Sig)$ is of the form $\star\nabla\ell_{\vec a}$, since $\star\mskip1mu\xi_i=\star\nabla x_i$.

\begin{proposition}\label{prop:dim-Lstar}
If $\Sig$ has finite Morse index and is not totally geodesic, then
\[
\dim_\R\bigl(\Lst(\Sig)\cap\Lsq(\Sig)\bigr)\le1.
\]
\end{proposition}

\begin{proof}
Consider the linear map
\[
T:\R^3\to\Lst(\Sig)\subset\HH^1(\Sig), \qquad T(\vec a)=\star\nabla\ell_{\vec a},
\]
which is onto, since $\star\mskip1mu\xi_i=T(E_i)$. As $\Lsq(\Sig)$ is a linear subspace of $\HH^1(\Sig)$
and $T$ is linear,
\[
S:=\bigl\{\vec a\in\R^3:T(\vec a)\in\Lsq(\Sig)\bigr\}
\]
is a linear subspace of $\R^3$ and
\begin{equation}\label{eq:image-S}
\Lst(\Sig)\cap\Lsq(\Sig)=T(S).
\end{equation}
Moreover $T$ is injective: by \eqref{eq:pointwise-Lstar}, $T(\vec a)=0$ if and only if $\vec a^\top\equiv0$, that is, $\vec a$ is everywhere normal to $\Sig$; since $\Sig$ is not totally geodesic, its unit normal $N$ is nonconstant and no nonzero $\vec a$ has this property. By \eqref{eq:image-S}, it therefore suffices to prove that $\dim S\le1$.

Let $\vec a\in S\setminus\{0\}$ and set $\xi=T(\vec a)$, which is nonzero by injectivity. By \eqref{eq:pointwise-Lstar} and $\xi\in L^2(\Sig)$,
\begin{equation}\label{eq:L2-normal}
\int_\Sig\bigl(|\vec a|^2-\langle\vec a,N\rangle^2\bigr)\dd\mu=\int_\Sig|\xi|^2\dd\mu<+\infty.
\end{equation}
By \S\ref{subsec:conformal}, finite Morse index gives finite total curvature, so along each end $\Ed$, the unit normal converges to a constant unit vector $n_\Ed$ (see \cite{Schoen}). Were $n_\Ed\ne\pm\vec a/|\vec a|$ for some end, then
\[
|\vec a|^2-\langle\vec a,N\rangle^2\ \longrightarrow\ |\vec a|^2-\langle\vec a,n_\Ed\rangle^2>0
\]
along that end, so the integrand of \eqref{eq:L2-normal} would be bounded below by a positive constant on a neighborhood of $\Ed$; as every end of a complete minimal surface has infinite area, this contradicts \eqref{eq:L2-normal}. Hence $n_\Ed=\pm\vec a/|\vec a|$ for every end.

Now let $\vec a,\vec a'\in S\setminus\{0\}$. Since $\Sig$ is noncompact, it has at least one end $\Ed$ and, by the previous paragraph, $n_\Ed=\pm\vec a/|\vec a|=\pm\vec a'/|\vec a'|$, which forces $\vec a'$ to be parallel to $\vec a$. Thus all nonzero vectors of $S$ are parallel, so $\dim S\le1$, and by \eqref{eq:image-S} and the injectivity of $T$,
\[
d:=\dim_\R\bigl(\Lst(\Sig)\cap\Lsq(\Sig)\bigr)=\dim T(S)=\dim S\le1.\qedhere
\]
\end{proof}

\begin{remark}\label{rem:d-vanishes}
The proof shows more: if $d=1$, with $S=\R\mskip1mu\vec a$, then every end of $\Sig$ has limiting unit normal $\pm\vec a/|\vec{a}|$. Contrapositively, $d=0$ as soon as the ends do not all share a common limiting normal, and in particular when some end is catenoidal; these are the situations listed in Theorem~\ref{thm:weak}. We do not know whether $d=1$ actually occurs; the statement of Theorem~\ref{thm:weak} is arranged so as not to depend on this. The totally geodesic case, where $T$ fails to be injective, is treated separately in \S\ref{subsec:proof-weak}.
\end{remark}

\section{Proof of the index estimates}\label{sec:proof}

We can now prove Theorems~\ref{thm:main} and~\ref{thm:weak}. We first record the dimension of $\Lsq(\Sig)$, which is the case $\rho\equiv1$ of Theorem~\ref{thm:dim-formula}.

\begin{corollary}\label{cor:dim-unweighted-fbms}
Let $\Sig$ be as in Section~\ref{sec:fbms}. Then
\begin{equation}\label{eq:dim-unweighted-fbms}
\dim_\R\Lsq(\Sig)=2g+k-1.
\end{equation}
\end{corollary}

\begin{proof}
The constant weight $\rho\equiv1$ is admissible, with radial profiles $\varrho_j\equiv1$ and local orders $N_j=0$ by Remark~\ref{rem:examples}(1). By Lemma~\ref{lem:conf-inv}(iii), the condition $\int_\Sig|\omega|^2\dd\mu<\infty$ is the same whether computed in the induced metric of $\Sig$ or in any conformally related metric, so $\Lsq(\Sig)$ is the space $\Lrho(\Sig)$ for this weight. Since all $N_j$ vanish, we have $\varepsilon=0$, and Theorem~\ref{thm:dim-formula} gives \eqref{eq:dim-unweighted-fbms}.
\end{proof}

\subsection{Proof of Theorem~\ref{thm:main}}\label{subsec:proof-main}

Let $\Sig$ be as in the statement, let $n:=\Ind(\Sig)<\infty$ and set
\[
V:=\Lstar(\Sig), \qquad \dim_\R V=2g+k+2\sum_{j=1}^r(d_j+1)-2\ \ge\ 3
\]
by Corollary~\ref{cor:fbms}. We use the logarithmic cut-off $\lcut_R$ of \S\ref{subsec:cutoff}. Suppose, for contradiction, that
\begin{equation}\label{eq:contradiction-hyp}
\dim_\R V>3n.
\end{equation}

Let $\phi_1,\dots,\phi_n$ and $W=\spann\{\phi_1,\dots,\phi_n\}$ be as in Theorem~\ref{thm:HS}; if $n=0$ then $W=\{0\}$ and $W^\perp=L^2(\Sig)$. Applying Corollary~\ref{cor:strict-positivity}, we may fix $R$ so large that
\begin{equation}\label{eq:main-negativity}
\cQ(\lcut_R\xi,\lcut_R\xi)<0\ \text{ for every }\ \xi\in V\setminus\{0\}.
\end{equation}

Consider the linear map
\[
\Phi_R\colon V\to\R^{3n},\qquad\Phi_R(\xi)=\left(\int_\Sig\lcut_R\mskip1mu u_i\mskip1mu\phi_j\dd\mu\right)_{1\le i\le3,\, 1\le j\le n},
\]
where $u_i=\langle\xi,E_i\rangle$. It is well defined: $\lcut_R u_i$ is smooth with compact support in $\Sig\cap\overline{B}_{R^2}$, on which $\xi$ is smooth and bounded, its poles lying at ambient infinity, away from $\supp\lcut_R$.

By \eqref{eq:contradiction-hyp} and the rank--nullity theorem, there exists $\xi\in\ker\Phi_R\setminus\{0\}$. For such a~$\xi$, we have $\lcut_R u_i\in C_c^\infty(\Sig)\cap W^\perp$ for $i=1,2,3$, so Theorem~\ref{thm:HS} gives 
\[
\cQ(\lcut_R u_i,\lcut_R u_i)\,\ge\,0\ \text{ for each }\ i,
\]
and summing over $i$,
\[
\cQ(\lcut_R\xi,\lcut_R\xi)\ =\ \sum_{i=1}^3\cQ(\lcut_R u_i,\lcut_R u_i)\ \ge\ 0,
\]
which contradicts \eqref{eq:main-negativity}. Therefore \eqref{eq:contradiction-hyp} fails, that is, $\dim_\R V\le3n$, which is \eqref{eq:main}. \qed

\begin{corollary}\label{cor:no-stable}
Under the hypotheses of Theorem~\ref{thm:main}, one has $\Ind(\Sig)\ge1$. In particular, there is no stable complete noncompact free boundary minimal surface 
with compact boundary in a domain $\Om$ whose boundary satisfies $H_{\dOm}\ge0$ along $\dSig$ and $H_{\dOm}>0$ somewhere on $\dSig$.
\end{corollary}

\begin{proof}
By Corollary~\ref{cor:fbms}, the right-hand side of \eqref{eq:main} is at least $3/3=1$.
\end{proof}

\subsection{Proof of Theorem~\ref{thm:weak}}\label{subsec:proof-weak}

Now $H_{\dOm}\ge0$ along $\dSig$, with no further assumption. We use the linear cut-off $\varphi_R$ of \S\ref{subsec:cutoff} and the unweighted space
\[
V:=\Lsq(\Sig),\ \text{ with }\ \dim_\R V=2g+k-1,
\]
by Corollary~\ref{cor:dim-unweighted-fbms}. Set $n:=\Ind(\Sig)$ and let $\phi_1,\dots,\phi_n$ and $W$ be as in Theorem~\ref{thm:HS}. We distinguish two
cases.

\subsubsection*{Case A: $\Sig$ is not totally geodesic}

Define the linear map
\[
\Phi\colon V\to\R^{3n},\qquad\Phi(\xi)=\left(\int_\Sig\phi_i\mskip1mu u_j\dd\mu\right)_{1\le i\le n,\, 1\le j\le3},\qquad u_j=\langle\xi,E_j\rangle,
\]
which is well defined because $\phi_i$ and $u_j$ are square integrable. By rank--nullity,
\begin{equation}\label{eq:rank-nullity-weak}
\dim_\R\ker\Phi\ \ge\ (2g+k-1)-3n.
\end{equation}
We claim that
\begin{equation}\label{eq:ker-in-Lstar}
\ker\Phi\ \subseteq\ \Lst(\Sig)\cap\Lsq(\Sig).
\end{equation}
Granting the claim, Proposition~\ref{prop:dim-Lstar} gives $\dim_\R\ker\Phi\le d\le1$, so \eqref{eq:rank-nullity-weak} yields $(2g+k-1)-3n\le d$, that is, $\Ind(\Sig)=n\ge(2g+k-1-d)/3$, which is \eqref{eq:weak}.

It remains to prove \eqref{eq:ker-in-Lstar}. Fix $\xi\in\ker\Phi$, so that each coordinate $u_j$ lies in $W^\perp$, and introduce the linear functional
\begin{equation}\label{eq:F-functional}
\mathcal F(Y):=2\int_\Sig\langle\nabla\xi,A\rangle\langle N,Y\rangle\dd\mu+\int_{\dSig}\langle Y,\cB\mskip1mu\xi\rangle\dd\sigma,
\end{equation}
defined for every smooth $Y\colon\Sig\to\R^3$ with $|Y|\in L^2(\Sig)$. The interior integral converges for such $Y$: indeed $|A|$ is bounded and $|\nabla\xi|\in L^2(\Sig)$ by Lemma~\ref{lem:aux}(a), so $\langle\nabla\xi,A\rangle\in L^2(\Sig)$, and the Cauchy--Schwarz inequality applies. The boundary integral converges because $\dSig$ is compact. Integrating \eqref{eq:Q} by parts and using $\cJ u_i=-2g_i\langle\nabla\xi,A\rangle$ from \eqref{eq:Jui}, one finds
\begin{equation}\label{eq:Q-equals-F}
\cQ(\xi,Y)=\mathcal F(Y)\ \text{ for every }\ Y\in C_c^\infty(\Sig;\R^3).
\end{equation}
By Theorem~\ref{thm:Ros}, proving \eqref{eq:ker-in-Lstar} amounts to proving $\langle\nabla\xi,A\rangle\equiv0$, which we do in three steps.

\medskip
\noindent\textbf{Step 1: $\mathcal F(Y)=0$ for every $Y\in C_c^\infty(\Sig;\R^3)$ whose coordinates lie in $W^\perp$.} Fix such a $Y$, choose $R$ with $\supp Y\subset B_R$, and let $\vec v_1,\dots,\vec v_n\in\R^3$ be chosen so that every coordinate of
\[
\zeta_R:=\varphi_R\Bigl(\xi+\sum_{k=1}^n\phi_k\vec v_k\Bigr)
\]
lies in $W^\perp$. Their existence, and the fact that $\vec v_k\to\vec0$ as $R\to+\infty$, follow from the solution to the linear system $b_{ij}(R)+\sum_ka_{ik}(R)v_{kj}=0$, where $b_{ij}(R):=\int_\Sig\varphi_R\phi_iu_j\dd\mu$ and $a_{ik}(R):=\int_\Sig\varphi_R\phi_i\phi_k\dd\mu$: by dominated convergence theorem, $a_{ik}(R)\to\delta_{ik}$ and $b_{ij}(R)\to\int_\Sig\phi_iu_j\dd\mu=0$, so the matrix $(a_{ik}(R))$ is invertible for $R$ large and the solution $\vec v_k=\vec v_k(R)$ tends to $\vec 0$.

Since $\varphi_R\equiv1$ on $\supp Y$, for every $t\in\R$, the field $\zeta_R+tY$ has compact support and all its coordinates lie in $W^\perp$, so Theorem~\ref{thm:HS} gives
\[
0\ \le\ \cQ(\zeta_R+tY,\zeta_R+tY)=\cQ(Y,Y)\mskip1mu t^2+2\mskip1mu\cQ(\zeta_R,Y)\mskip1mu t+\cQ(\zeta_R,\zeta_R)
\]
for every $t\in\R$. A quadratic polynomial in $t$ is nonnegative on $\R$ if and only if its leading coefficient is nonnegative and its discriminant is nonpositive, so
\begin{equation}\label{eq:CS}
\cQ(Y,Y)\ \ge\ 0\quad \text{ and }\quad \cQ(\zeta_R,Y)^2\ \le\ \cQ(Y,Y)\mskip1mu\cQ(\zeta_R,\zeta_R).
\end{equation}

We compute the two sides of the second inequality in \eqref{eq:CS}. For the left side, all the integrands defining $\cQ(\cdot,Y)$ are supported in $\supp Y$, where $\varphi_R\equiv1$, so
\[
\cQ(\zeta_R,Y)=\cQ(\xi,Y)+\sum_{k=1}^n\cQ(\phi_k\vec v_k,Y).
\]
Moreover, integrating by parts and using $\cJ\phi_k+\lambda_k\phi_k=0$ and $\cB\mskip1mu\phi_k=0$,
\[
\begin{aligned}
\cQ(\phi_k\vec v_k,Y)&=-\int_\Sig\langle Y,\vec v_k\rangle\cJ\phi_k\dd\mu+\int_{\dSig}\langle Y,\vec v_k\rangle\cB\mskip1mu\phi_k\dd\sigma\\[4pt]
&=\lambda_k\int_\Sig\langle Y,\vec v_k\rangle\phi_k\dd\mu=0,
\end{aligned}
\]
the last integral vanishing because the coordinates of $Y$ are $L^2$-orthogonal to $\phi_k$. Hence $\cQ(\zeta_R,Y)=\cQ(\xi,Y)=\mathcal F(Y)$ by \eqref{eq:Q-equals-F}, a quantity independent of $R$.

For the right side, we expand
\[
\cQ(\zeta_R,\zeta_R)=\mathrm{(I)}+2\mskip1mu\mathrm{(II)}+\mathrm{(III)},
\]
\[
\mathrm{(I)}=\cQ(\varphi_R\xi,\varphi_R\xi),\qquad\mathrm{(II)}=\sum_{k=1}^n\cQ(\varphi_R\xi,\varphi_R\phi_k\vec v_k),
\]
\[
\mathrm{(III)}=\sum_{k,\ell=1}^n\cQ(\varphi_R\phi_k\vec v_k,\varphi_R\phi_\ell\vec v_\ell),
\]
and estimate the three terms as $R\to+\infty$.

\emph{Term} (I). By Lemma~\ref{lem:key-identity} applied to $\xi$ and to the linear cut-off $\varphi_R$, together with \eqref{eq:linear-gradient},
\[
\mathrm{(I)}=\int_\Sig|\nabla\varphi_R|^2|\xi|^2\dd\mu-\int_{\dSig}H_{\dOm}|\xi|^2\dd\sigma\ \le\ \frac{C^2}{R^2}\|\xi\|_{L^2}^2-\int_{\dSig}H_{\dOm}|\xi|^2\dd\sigma,
\]
so that
\[
\limsup_{R\to+\infty}\mathrm{(I)}\ \le\ -\int_{\dSig}H_{\dOm}|\xi|^2\dd\sigma\ \le\ 0,
\]
where we used that $H_{\dOm}\ge0$ along $\dSig$.

\emph{Terms} (II) \emph{and} (III). Integrating by parts and using $\cJ\phi_k+\lambda_k\phi_k=0$ and $\cB\mskip1mu\phi_k=0$,
\[
\begin{aligned}
\cQ(\varphi_R\xi,\varphi_R\phi_k\vec v_k)\,={}&\lambda_k\int_\Sig\varphi_R^2\langle\xi,\vec v_k\rangle\phi_k\dd\mu-\int_\Sig(\varphi_R\Delta\varphi_R)\langle\xi,\vec v_k\rangle\phi_k\dd\mu\\
&-2\int_\Sig\varphi_R\langle\xi,\vec v_k\rangle\langle\nabla\varphi_R,\nabla\phi_k\rangle\dd\mu,
\end{aligned}
\]
and similarly for the terms of (III). Now $\varphi_R$, $\varphi_R\Delta\varphi_R$ and $|\nabla\varphi_R|$ are bounded uniformly in $R$, while $|\xi|$, $\phi_k$ and $|\nabla\phi_k|$ are square integrable by Lemma~\ref{lem:aux}; hence each of the three integrals above is bounded uniformly in $R$ by the Cauchy--Schwarz inequality. Since $\vec v_k\to\vec 0$, both (II) and (III) tend to $0$.

Combining $\limsup_{R\to+\infty}\cQ(\zeta_R,\zeta_R)\le0$ and \eqref{eq:CS} gives
\[
\mathcal F(Y)^2=\cQ(\zeta_R,Y)^2\le\cQ(Y,Y)\mskip1mu\cQ(\zeta_R,\zeta_R)\ \text{ for every large }\ R,
\]
with $\cQ(Y,Y)\ge0$ fixed. Letting $R\to+\infty$ yields $\mathcal F(Y)^2\le0$, that is, $\mathcal F(Y)=0$.

\medskip
\noindent\textbf{Step 2: $\mathcal F(Z)=0$ for every smooth $Z$ with $|Z|\in L^2(\Sig)$ whose coordinates lie in $W^\perp$.} Given such a $Z$, choose $\vec w_1,\dots,\vec w_n\in\R^3$, depending on $R$ and
tending to $\vec 0$, so that the compactly supported field
\[
Y_R:=\varphi_R\Bigl(Z+\sum_{k=1}^n\phi_k\vec w_k\Bigr)
\]
has all its coordinates in $W^\perp$; the argument is the one used for $\zeta_R$ in Step~1. By Step~1, $\mathcal F(Y_R)=0$, that is, using $\varphi_R\equiv1$ near $\dSig$,
\[
\begin{aligned}
&2\int_\Sig\varphi_R\langle\nabla\xi,A\rangle\langle N,Z\rangle\dd\mu+\int_{\dSig}\langle Z,\cB\mskip1mu\xi\rangle\dd\sigma\\[4pt]
&\qquad=-2\sum_{k=1}^n\int_\Sig\varphi_R\phi_k\langle\nabla\xi,A\rangle\langle N,\vec w_k\rangle\dd\mu-\sum_{k=1}^n\int_{\dSig}\phi_k\langle\vec w_k,\cB\mskip1mu\xi\rangle\dd\sigma.
\end{aligned}
\]
On the right-hand side, each interior integral is bounded uniformly in $R$, by the Cauchy--Schwarz inequality, and each boundary integral is bounded because $\dSig$ is compact; since $\vec w_k\to\vec 0$, the whole right-hand side tends to $0$. On the left-hand side, the integrand of the interior integral is dominated by $2|\langle\nabla\xi,A\rangle||Z|\in L^1(\Sig)$, so dominated convergence applies as $\varphi_R\to1$. Passing to the limit gives $\mathcal F(Z)=0$.

\medskip
\noindent\textbf{Step 3: $\mathcal F(\phi_\ell\vec a)=0$ for every $\ell\in\{1,\dots,n\}$ and every $\vec a\in\R^3$.} Write $v:=\langle\xi,\vec a\rangle$, so that $v\in W^\perp\subset L^2(\Sig)$, since the coordinates of $\xi$ lie in $W^\perp$. By Theorem~\ref{thm:Ros},
\begin{equation}\label{eq:Ros-v}
2\langle\nabla\xi,A\rangle\langle N,\vec a\rangle=-\cJ v=-\Delta v-|A|^2v.
\end{equation}
Consider the truncated expression
\[
\mathcal F_R:=2\int_\Sig\varphi_R\phi_\ell\langle\nabla\xi,A\rangle\langle N,\vec a\rangle\dd\mu+\int_{\dSig}\phi_\ell\langle\vec a,\cB\mskip1mu\xi\rangle\dd\sigma,
\]
which tends to $\mathcal F(\phi_\ell\vec a)$ as $R\to+\infty$, by dominated convergence in the interior integral, the integrand being dominated by $|\langle\nabla\xi,A\rangle|\mskip1mu|\phi_\ell|\in L^1(\Sig)$. We claim that $\mathcal F_R\to0$.

By \eqref{eq:Ros-v}, the first term of $\mathcal F_R$ equals $-\int_\Sig\varphi_R\phi_\ell(\Delta v+|A|^2v)$. For the second term, Green's identity applied to the pair $\varphi_R\phi_\ell$ and $v$, together with $\varphi_R\equiv1$ near $\dSig$ and with $\p_\nu\phi_\ell=\II_{\dOm}(N,N)\phi_\ell$, which is $\cB\mskip1mu\phi_\ell=0$, gives
\[
\begin{aligned}
\int_\Sig\bigl(\varphi_R\phi_\ell\Delta v-v\mskip1mu\Delta(\varphi_R\phi_\ell)\bigr)\dd\mu&=\int_{\dSig}\bigl(\phi_\ell\p_\nu v-v\mskip1mu\II_{\dOm}(N,N)\phi_\ell\bigr)\dd\sigma\\[4pt]
&=\int_{\dSig}\phi_\ell\langle\vec a,\cB\mskip1mu\xi\rangle\dd\sigma,
\end{aligned}
\]
where we used $\p_\nu v=\langle\vec a,D_\nu\xi\rangle$ and the definition \eqref{eq:B-vector} of $\cB\mskip1mu\xi$. Adding the two contributions, the terms containing $\Delta v$ cancel and
\[
\mathcal F_R=-\int_\Sig\varphi_R\phi_\ell|A|^2v\dd\mu-\int_\Sig v\mskip1mu\Delta(\varphi_R\phi_\ell)\dd\mu.
\]
Expanding $\Delta(\varphi_R\phi_\ell)=\phi_\ell\Delta\varphi_R+2\langle\nabla\varphi_R,\nabla\phi_\ell\rangle+\varphi_R\Delta\phi_\ell$ and using
\[
\Delta\phi_\ell=-(|A|^2+\lambda_\ell)\phi_\ell,
\]
the terms containing $|A|^2$ cancel as well, and we are left with
\begin{equation}\label{eq:FR-three-terms}
\mathcal F_R=\lambda_\ell\int_\Sig\varphi_R\mskip1mu v\mskip1mu\phi_\ell\dd\mu-\int_\Sig v\mskip1mu\phi_\ell\mskip1mu\Delta\varphi_R\dd\mu-2\int_\Sig v\mskip1mu\langle\nabla\varphi_R,\nabla\phi_\ell\rangle\dd\mu.
\end{equation}
We treat the three terms of \eqref{eq:FR-three-terms} in turn. The first tends to $\lambda_\ell\int_\Sig v\phi_\ell\dd\mu=0$, by dominated convergence and the orthogonality $v\in W^\perp$. For the second, integrating by parts and using that $\nabla\varphi_R$ vanishes near $\dSig$,
\[
\left|\int_\Sig v\mskip1mu\phi_\ell\mskip1mu\Delta\varphi_R\dd\mu\right|=\left|\int_\Sig\bigl\langle\nabla(v\phi_\ell),\nabla\varphi_R\bigr\rangle\dd\mu\right|\le\frac{C}{R}\int_\Sig|\nabla(v\phi_\ell)|\dd\mu\ \longrightarrow\ 0,
\]
because $|\nabla(v\phi_\ell)|\in L^1(\Sig)$ by Lemma~\ref{lem:aux}(c). For the third, by the Cauchy--Schwarz inequality and Lemma~\ref{lem:aux}(b), 
\[
\left|2\int_\Sig v\langle\nabla\varphi_R,\nabla\phi_\ell\rangle\dd\mu\right|\le\frac{2C}{R}\mskip1mu\|v\|_{L^2}\mskip1mu\|\nabla\phi_\ell\|_{L^2}\ \longrightarrow\ 0.
\]
Hence $\mathcal F_R\to0$ and $\mathcal F(\phi_\ell\vec a)=0$.

\medskip
\noindent\textbf{Conclusion of Case A.}
Let $Z\in C_c^\infty(\Sig;\R^3)$ be arbitrary and decompose each coordinate as $z_i=z_i^\perp+\sum_{\ell=1}^n c_{i\ell}\phi_\ell$ with $c_{i\ell}=\int_\Sig z_i\phi_\ell\dd\mu$ and $z_i^\perp\in W^\perp$; setting
$\vec c_\ell=(c_{1\ell},c_{2\ell},c_{3\ell})$, we get
\[
Z=Z^\perp+\sum_{\ell=1}^n\phi_\ell\vec c_\ell,
\]
where $Z^\perp$ is smooth, square integrable and has coordinates in $W^\perp$. By Steps~2 and~3 and the linearity of $\mathcal F$,
\[
\mathcal F(Z)=\mathcal F(Z^\perp)+\sum_{\ell=1}^n\mathcal F(\phi_\ell\vec c_\ell)=0.
\]
Choosing $Z=\psi N$ with $\psi\in C_c^\infty(\Sig)$ supported in the interior of $\Sig$, the boundary term of \eqref{eq:F-functional} drops and $\langle N,Z\rangle=\psi$, so
\[
\int_\Sig\langle\nabla\xi,A\rangle\mskip1mu\psi\dd\mu=0\ \text{ for every such }\ \psi,
\]
whence $\langle\nabla\xi,A\rangle\equiv0$ on $\Sig$. Since $\Sig$ is not totally geodesic, Theorem~\ref{thm:Ros} gives $\xi\in\Lst(\Sig)$, and since $\xi\in\Lsq(\Sig)$ by hypothesis, this proves \eqref{eq:ker-in-Lstar} and completes Case~A.

\subsubsection*{Case B: $\Sig$ is totally geodesic}

Then $A\equiv0$, the genus is $g=0$, and $\Sig$ is an immersed planar domain $\Om'\subset\R^2$ with $k$ smooth compact boundary components. The second variation \eqref{eq:Q} reduces to
\[
\cQ(u,u)=\int_{\Om'}|\nabla u|^2\dd\mu-\int_{\p\Om'}\II_{\dOm}(N,N)u^2\dd\sigma.
\]
If $H_{\dOm}>0$ at some point of $\dSig$, Theorem~\ref{thm:main} applies so that \eqref{eq:main} implies \eqref{eq:weak}. Otherwise, if $H_{\dOm}\equiv 0$ along $\dSig$, the free boundary condition gives 
\[
\II_{\dOm}(N,N)=-\II_{\dOm}(\tau,\tau)=-k_g,
\]
where $k_g$ is the geodesic curvature of $\dSig$ in $\Sig$. Hence $\Ind(\Sig)$ equals the index of the intrinsic Robin problem
\[
\Delta u+\lambda u=0\ \text{ in }\ \Om',\qquad\p_\nu u+k_gu=0\ \text{ on }\ \p\Om'.
\]
By Theorem~1 of the second-named author's work \cite{Mendes}, either $\Sig$ is stable and is homeomorphic to a compact disk, which in our case cannot happen because we are assuming that $\Sig$ is noncompact, or $\Sig$ is unstable and 
\[
\Ind(\Sig)\ \ge\ \max\left\{1,\left\lceil\frac{k-2}{2}\right\rceil\right\}.
\]
In the latter case, we compare with $(2g+k-1)/3=(k-1)/3$, recalling $g=0$: for $k\le4$, one has $(k-1)/3\le1$, so the bound $\Ind(\Sig)\ge1$ suffices; while for $k>4$, one has $\lceil(k-2)/2\rceil\ge(k-2)/2>(k-1)/3$. In either case,
\[
\Ind(\Sig)\ge\frac{2g+k-1}{3}\ge\frac{2g+k-1-d}{3},
\]
which is \eqref{eq:weak}. \qed

\section*{Acknowledgments}
The $L^2$ part of this work
is contained in the Ph.D. thesis of I.\,R.\,S.\ \cite{Santos-thesis}, written under the supervision of M.\,C.\ and defended at the Federal University of Alagoas on May 25, 2026. 
The third author thanks the Graduate Program in Mathematics of UFAL for its support and hospitality during their Ph.D. studies.
This work was carried out within the CAPES--Math AmSud project \emph{New Trends in Geometric Analysis}
(CAPES, Grant 88887.985521/ 2024-00).
M.\,C.\ and A.\,M.\ were partially supported by CNPq (Grants: 311136/2023-0 to M.\,C.; 309867/2023-1 and 445723/2025-4 to A.\,M.); I.\,R.\,S.\  was financed in part by the Coordenação de Aperfeiçoamento de Pessoal de Nível Superior - Brasil (CAPES) - Finance Code 001

\section*{Statement on the use of generative AI}
During the preparation of this work the authors used the large language
model Claude (Anthropic) as an assistant, in the following ways. First, the
model was used to read successive drafts and to point out possible gaps in
the arguments; all corrections and alternative arguments suggested by the
model were checked in full by the authors before being adopted. Second, the
model was used to draft or rewrite expository passages, in particular in the
introduction and in several remarks, starting from the authors' notes; these
passages were then revised by the authors. Third, the model assisted with the
\LaTeX{} source and with the consistency of notation and terminology, and it
was used to locate related results in the literature; every reference cited
was consulted directly by the authors. The research questions, the main
results and the strategy of the proofs are due to the authors, who take full
responsibility for the content of this paper, including its correctness.

\bibliographystyle{amsplain}
\bibliography{bibliography-v2}

\end{document}